\documentclass[11pt]{article}
\usepackage{amsfonts}
\usepackage{CJK}
\usepackage{graphics}
\usepackage{indentfirst}
\usepackage{cite}
\usepackage{latexsym}
\usepackage{amsmath}
\usepackage{amsthm}
\usepackage{amssymb}
\usepackage[dvips]{epsfig}
\usepackage{amscd}

\usepackage{color}

\newtheorem{theorem}{Theorem}[section]
\newtheorem{remark}{Remark}[section]
\newtheorem{definition}{Definition}[section]
\newtheorem{lemma}[theorem]{Lemma}
\newtheorem{pro}{Proposition}[section]

\newcommand{\ti}{\tilde}

\renewcommand{\div}{ {\rm div }  }

\newcommand{\na}{\nabla }

\newcommand{\pa}{\partial}
\renewcommand{\r}{\mathbb{R}}

\newcommand{\bl}{\begin{lemma}}
\newcommand{\el}{\end{lemma}}
\newcommand{\et}{\end{theorem}}

\newcommand{\curl}{{\rm curl} }
\newcommand{\te}{\theta}

\newcommand{\al}{\alpha}
\newcommand{\de}{\delta}
\newcommand{\ve}{\varepsilon}
\newcommand{\la}{\label}

\newcommand{\p}{p(\rho)  }

\newcommand{\ka}{\kappa}

\newcommand{\bn}{\begin{eqnarray}}
\newcommand{\en}{\end{eqnarray}}
\newcommand{\bnn}{\begin{eqnarray*}}
\newcommand{\enn}{\end{eqnarray*}}

\newcommand{\bnnn}{\begin{eqnarray*}}
\newcommand{\ennn}{\end{eqnarray*}}

\newcommand{\ba}{\begin{aligned}}
\newcommand{\ea}{\end{aligned}}
\newcommand{\be}{\begin{equation}}
\newcommand{\ee}{\end{equation}}
\def\O{\Omega}
\def\p{\partial}
\def\norm[#1]#2{\|#2\|_{#1}}

\def\o{\omega}

\newcommand{\n}{\rho}
\newcommand{\thatsall}{\hfill$\Box$}

\def\la{\label}

\def\na{\nabla}

\def\p{\partial}
\def\norm[#1]#2{\|#2\|_{#1}}

\def\o {\omega}

\def\O{B_r}

\makeatletter      
\@addtoreset{equation}{section}
\makeatother       
\title{On Existence and Blowup Criterion of Strong Solutions to Cauchy Problem of 2D  Full Compressible Navier-Stokes System with Vacuum}
\author
{ Xue Wang\thanks{Email addresses:  xuewa@amss.ac.cn} \\[3mm]
\\{\normalsize    School of Mathematical Sciences,}\\
{\normalsize  University of Chinese Academy of Sciences, Beijing 100049, P. R. China}}
\date{}

\begin{document}
\maketitle

\begin{abstract}
This paper investigates the Cauchy problem of two-dimensional  full compressible Navier-Stokes system  with  density and temperature vanishing  at infinity. For the strong solutions, some a priori weighted $L^2(\r^2)$-norm of the gradient of velocity is obtained by the techniques of $A_p$ weights and cancellation of singularity. Based on this key weighted estimate and some basic weighted analysis of velocity and temperature,   we  establish the local existence and uniqueness of strong solutions with initial vacuum by means of a two-level approximation scheme. Meanwhile, for $q_1$ and $q_2$ as in \eqref{q7}, the $L^\infty_tL^{q_1}_x$-norm of velocity and the $L^2_tL^{q_2}_x$-norm of temperature are derived originally.  Moreover, we obtain a blowup criterion only in terms of the temporal integral of the maximum norm of divergence of velocity, which is independent of the temperature.
\end{abstract}

\textbf{Keywords}:  full compressible Navier-Stokes  system;    $A_p$ weights; cancellation of singularity;  blowup criterion; vacuum

\section{Introduction}
We consider the  full compressible Navier-Stokes system in two-dimensional (2D) spatial domain $\Omega\subset\r^2$, which is read as follows:
\be \la{q1}  \begin{cases}
\n_t+{\rm div} (\n u)=0,\\
(\n u)_t+\div(\n u\otimes u)-\mu\Delta u-(\mu+\lambda)\na{\rm div} u+\na P=0, \\
c_\upsilon[(\n \te)_t+\div(\n \te u)]-2\mu |\mathfrak{D}(u)|^2-\lambda (\div u)^2+P\div u-\ka\Delta\te=0,
\end{cases}\ee
where $t\ge 0$, $x=(x_1,x_2)\in \Omega$, and $\n$, $u=\left(u^1,u^2\right),$ $\te$, $P=R\n\te(R>0),$ and $\mathfrak{D}(u)  = (\nabla u + (\nabla u)^{\rm tr})/2$ represent respectively the fluid density, velocity, absolute temperature, pressure, and   deformation tensor.
The constant viscosity coefficients $\mu$ and $\lambda$ satisfy the physical restrictions:
\be\la{h3} \mu>0,\quad \mu + \lambda\ge 0.\ee
Positive constants $c_\upsilon$ and $\ka$ are, respectively, the heat capacity  and  the ratio of  the heat conductivity coefficient over the heat capacity.

Let $\Omega=\r^2$, we impose \eqref{q1} the following   initial data:
\be\la{q2}(\n,\n u,\n\te)(x,0)=(\n_0,\n_0 u_0,\n_0\te_0)(x),\quad x\in\r^2,\ee
and far field behavior (in some weak sense):
\be\la{q3}(\n,u,\te)(x,t)\rightarrow(0,0,0),\quad \text{as}\quad |x|\rightarrow \infty,\,t>0.\ee

There is huge amounts of literature on the well-posedness of solutions to the compressible Navier-Stokes equations.  For the case that the initial density is  strictly away from vacuum, the local well-posedness theory is known in \cite{se1,k1,ss}. In 1980s, Matsumura-Nishida \cite{M1} first proved the global existence of classical solutions for the initial data close to a non-vacuum equilibrium in some Sobolev space $H^s$. Later, Hoff \cite{Hof1} showed the global weak solutions with strictly positive initial density and temperature for discontinuous initial data. When the vacuum state  is allowed, the significant breakthrough is due to Lions \cite{L2} (improved by Feireisl \cite{fe1}), where he  established the global existence of  weak solutions with finite energy and arbitrary initial data, under the condition that the adiabatic exponent $\gamma$ is  suitably large. If the initial data satisfy additional  regularity and compatibility conditions,  the local existence and uniqueness of  strong and classical solutions  were  proved in  \cite{K2,hua,lxz}. Recently,  Huang-Li-Xin\cite{hulx}, Huang-Li \cite{H-L},  Cai-Li \cite{C-L}, and Li-L{\small \"u}-Wang\cite{llw}, respectively, obtained the global classical solutions to 3D Cauchy problem or slip boundary problem for barotropic flows or heat-conducting fluids \eqref{q1}  with the initial data which are of small energy but possibly large oscillations.  For  the 2D Cauchy problem, Li-Liang \cite{liliang} and Li-Xin \cite{lx} gave the local and global existence  of  classical solutions to the barotropic compressible Navier-Stokes equations.  However, there are still a lot of open problems.

On the other hand, in the process of confirming whether the global solutions exist or not, it's interesting to investigate the mechanism of blowup and the structure of possible singularities. Xin \cite{xin} first showed that if the initial density has compact support, any smooth solutions to the Cauchy problem of the full compressible Navier-Stokes system without heat conduction blow up in finite time. For more information on the blowup criteria of  compressible flow,  we refer to \cite{fj,fz,hlw,hlx,hlx2,swz2,wy,r} and the references therein. In particular, for 2D full  compressible Navier-Stokes system in  bounded domain $\Omega$, Wang \cite{wy} derived a blowup criterion regarding the divergence of the velocity only, i.e.,
\be\notag\lim_{T\rightarrow T_\ast}\|\div u\|_{L^1(0,T;L^\infty(\Omega))}=\infty,\ee
where $T^\ast<\infty$ is the maximal time of existence of a strong solution.

The aim of this paper is to establish the local existence of  strong  solutions  to the 2D Cauchy problem of \eqref{q1} and a corresponding blowup criterion.
Before stating the main results, we first explain the notations and conventions used throughout the paper except the appendix. 
We set
$$\int f dx=\int_{\r^2}f dx.$$
For $1\le p\le \infty$ and integer $k\geq 0$, the standard Lebesgue and Sobolev  spaces are denoted by:
$$L^p=L^p(\r^2),\quad W^{k,p}=W^{k,p}(\r^2),\quad H^{k}=W^{k,2}.$$
Next, we give the definition of strong solutions  as follows:
\begin{definition}
If all the derivatives involved in (\ref{q1}) for $(\rho, u, \te)$ are regular distributions, and the equations (\ref{q1}) hold almost every where in $\r^2 \times (0,T)$, then $(\rho, u, \te)$ is called a strong solution to (\ref{q1}).
\end{definition}

Now, we state the first main result in this paper concerning the local existence  of strong solutions.
\begin{theorem}\la{th1}
Let $\eta_0$ be a positive constant and
\be\notag\bar x\triangleq (e+|x|^2)^{1/2}\log^{1+\eta_0}(e+|x|^2).\ee
For constants $q>2$ and $a>1$,  assume that the initial data $(\n_0\geq 0,u_0,\te_0\geq 0)$ satisfy
\be\ba\la{q4}
\bar x^a\n_0\in L^1\cap H^1\cap W^{1,q},\,(\sqrt{\n_0} u_0,\sqrt{\n_0}\te_0)\in L^2,\,\na u_0\in H^1,\,\na\te_0\in L^2,
\ea\ee
and the compatibility condition
\be\la{co1}-\mu\Delta u_0-(\mu+\lambda)\na\div u_0+R\na(\n_0\te_0)=\n_0^{1/2}g,\ee
with  $g\in L^2$.
Then there exists a positive time  $T^\ast$ such that the Cauchy problem (\ref{q1}) (\ref{q2}) (\ref{q3}) admits a unique strong solution $(\n\geq 0,u,\te\geq 0)$ on $\r^2\times(0,T^\ast]$ satisfying
\be\ba\la{q5}
\begin{cases}
\n\in C([0,T^\ast];L^1\cap H^1\cap W^{1,q}),\,\bar x^{a}\n\in L^\infty(0,T^\ast;L^1\cap H^1\cap W^{1,q}),\\
\na u\in L^\infty(0,T^\ast;H^1)\cap L^2(0,T^\ast;W^{1,q}), \quad u\in L^\infty(0,T^\ast;L^{q_1}),\\ 
\na\te,\,\sqrt t\na^2\te\in L^\infty(0,T^\ast;L^2)\cap L^2(0,T^\ast;H^{1}), \quad \te\in L^2(0,T^\ast;L^{q_2}),\\ 
\sqrt{\n}u,\,\sqrt{\n}\te,\,\sqrt{\n}u_t,\,\sqrt t\sqrt{\n}\te_t,\,\na u\bar x^{\al/2}\in L^\infty(0,T^\ast;L^2),\\ 
\na u_{t},\,\sqrt{\n}\te_t,\,\sqrt t\na\te_t\in L^2(\r^2\times(0,T^\ast)),
\end{cases}
\ea\ee
where \be\la{q7}\al\triangleq\min\{a/2,1\},\quad q_1=4/\al,\quad q_2=6/(2\al-1).\ee
\end{theorem}

Based on the strong solutions obtained in Theorem \ref{th1}, we have the following blowup criterion.
\begin{theorem}\la{th2}
Under the conditions of Theorem \ref{th1}, let $(\n,u,\te)$ be the strong solution to the Cauchy problem (\ref{q1}) (\ref{q2}) (\ref{q3})  obtained in Theorem \ref{th1}, which satisfies (\ref{q5}). If $T_\ast<\infty$ is the maximal time of existence, then
\be\la{qq6}\lim_{T\rightarrow T_\ast}\|\div u\|_{L^1(0,T;L^\infty)}=\infty.\ee
\end{theorem}

\begin{remark}
Theorem \ref{th1} is the first result concerning the local existence of strong solutions with vacuum to the Cauchy problem of 2D full compressible Navier-Stokes system. As is well-known, compared to the 3D Cauchy problem, the essential difficulty for 2D case is  lack of the integrability of velocity and temperature. Here, by virtue of the techniques of harmonic analysis and some delicate analysis, we solve this problem and obtain the integrability of velocity and temperature successfully, which is a major breakthrough in comparison to the result in \cite{liliang} for isentropic flow. In fact, adopting the method in this paper, we can get the similar result for isentropic flow.
\end{remark}

\begin{remark}
If the initial data satisfy additional regularity, the strong solution obtained in Theorem \ref{th1} becomes a classical one.
\end{remark}


We now comment on the analysis of this paper.
For the local existence theory, as indicated in \cite{K2,liliang}, the methods for bounded domains in $\r^2$ are not suitable for the  Cauchy problem, since the $L^p$-norms of $u$ and $\te$ are unavailable for some $p\geq1$. To overcome  this difficulty caused by the unbounded domain, we mainly adopt the idea in \cite{liliang}.  If initial density $\n_0$ decays for large $|x|$,  $\n(\cdot,t)$  decays at the same rate, which leads to   $\n u$ and $\n\te$  decay faster than $u$ and $\te$ themselves. Then,  making use of the key Poincar$\acute{e}$ inequality (see Lemmas \ref{w3}, \ref{w5}), $u\bar x^{-\eta}$ and $\te\bar x^{-\eta}$ $(\eta>0)$  will be bounded instead of $u$ and $\te$. 

Next, compared to the isentropic case,  the main difficulty lies in dealing with the terms only involving the temperature and the dissipative function $F(\na u)$ given in \eqref{poi}, for example, $\int F(\na u)\te dx$. In fact, these terms require that $\te$ or $\na u$ must have faster decay rate. Here, we consider that  $|\na u|\bar x^k\,(k>0)$ is controlled, which plays a crucial role in this paper. To do so, we take a  two-level approximation strategy. First, considering the approximation system \eqref{a1} with a damping term $\de \te$, we extend the standard classical solutions  in bounded domains to the whole space. 
Then 
we turn to the key step of the paper: the analysis on the whole space. 

Our central observation is the following cancellation property: 
\begin{equation}\la{q8}
\int_{\mathbb{R}^2}(\rho\dot{u})(y,t)dy=0,
\end{equation}
which is only available in $\mathbb{R}^2$, and  will lead to the extra order of decay on $\nabla u$.

To be precise, based on the structure of \eqref{a1}$_2$ and the classical theory of singular integrals (refers to Stein \cite{stein1}), we  divide $\na u$ into the ``good'' part and the ``bad'' part:
\be\notag\p_i u\triangleq \mathcal{T}_i P+K_i\ast (\n\dot u),\quad i=1,2,\ee
 where in ``good'' part, $\mathcal{T}_i=-\frac{1}{2\mu+\lambda}\p_i(-\Delta)^{-1}\na\,(i=1,2)$ is the singular integral operator of convolution type with Calderon-Zygmund kernel which is bounded in $L^p\,(1<p<\infty)$ and the ``bad'' part is in the form of 
\begin{equation}\la{q9}
\int_{\mathbb{R}^2}K_i(x-y)(\rho\dot{u})(y,t)dy,
\end{equation}
with the  kernel $K_i(x)$ satisfying
\be\notag|K_i(x)|\le C|x|^{-1},\quad|\na K_i(x)|\le C|x|^{-2},\quad i=1,2.\ee
With the help of the well-known $A_p$ weights theory (see Stein \cite{stein}), the ``good'' part has suitable decay rate. While the ``bad'' part seems to be troublesome. Fortunately, making use of the cancellation condition \eqref{q8}, the ``bad" part \eqref{q9}  can turn into
$$\int_{\mathbb{R}^2} (K_i(x-y)-K_i(x))(\rho\dot{u})(y,t)dy,$$
where the singularity at infinity is actually canceled out and
 we thus succeed  in obtaining  the desired higher decay rate independent of $\de$ of the ``bad'' part (see \eqref{k1}). This combined with the estimate on ``good'' part guarantees  the weighted-$L^2(\r^2)$ estimate on $\na u$.


Moreover, using this key estimate and  Caffarelli-Kohn-Nirenberg inequality (see Lemma \ref{w4}), we can arrive at the integrability of velocity and other necessary estimates.
With all the desired a priori estimates independent $\de$ at hand, by letting $\de$ tend to zero, the standard compactness arguments then show the local existence of strong solutions to the Cauchy problem \eqref{q1} \eqref{q2} \eqref{q3}. In addition, the method in Lemma \ref{lee0} for velocity implies the integrability of temperature as well.

For the blowup criterion, the key issue is to estimate the $L^\infty_tL^2_x$-norm of the gradient of velocity. After standard calculations, the problem is turned into estimating the $L^2_tL^\infty_x$-norm of velocity (see \eqref{c3}). Unlike the case of the bounded domains in \cite{wy}, where the maximum of velocity $u$ is easily bounded by  $\na u$ according to the Sobolev inequality,  the Cauchy problem is more complicated. Indeed, by using the approach in Lemma \ref{lee0} again and some delicate estimates, we can control the maximum of velocity by some suitable terms.

The rest of this paper is organized as follows: Section \ref{sec2} collects  some basic facts and inequalities  needed in later analysis. In Section \ref{sec3},  based on the approximate  solutions to the Cauchy problem of \eqref{a1} and some uniform a priori estimates (with respect to $\de$), we show the first main result Theorem \ref{th1}. In Section \ref{sec4},  the blowup criterion Theorem \ref{th2} is proved. Section \ref{sec5} is  devoted to deriving the strong solutions to the Cauchy problem of  approximation system \eqref{a1}.

\section{Preliminaries}\la{sec2}
The whole analysis of Theorem \ref{th1} is based on the following approximation system with a damping of $\te$:
\be \la{a1}  \begin{cases}
\n_t+{\rm div} (\n u)=0,\\
\n (u_t+ u\cdot \na u)-\mu\Delta u-(\mu+\lambda)\na{\rm div}u+\na P=0, \\
c_\upsilon\n (\te_t+u\cdot \na \te)-F(\na u)+P\div u-\ka\Delta\te+\de \te=0,
\end{cases}\ee
where $\de$ is a positive constant  and
\be\la{poi} F(\na u)\triangleq 2\mu |\mathfrak{D}(u)|^2+\lambda (\div u)^2.\ee

To begin with, we consider the initial-boundary-value problem to \eqref{a1} on $\O\times(0,T]$ with $T>0$ and $B_r\triangleq\{x\in\r^2||x|<r\}\,(r>0)$.
Indeed,  imposing \eqref{a1} the initial data:
\be\la{bbbb} (\n,u,\te)(x,0)=(\n_0,u_0,\te_0), \quad x\in\O,\ee
and boundary conditions:
\be\la{bbb} u=0,\quad \na\te\cdot n=0 \quad \mbox{ on }\pa B_r,\ee
 we  have the following  local existence result of classical solutions  with strictly positive initial density to \eqref{a1}--\eqref{bbb}, which can be proved by the similar way to that in \cite{K2}.
\begin{lemma}\la{w1}
Assume  that
 $(\n_0,u_0,\te_0)$ satisfy
 \be \la{2.1}
\inf\limits_{x\in\O}\n_0(x)> 0,\,\,(\n_0,\te_0)\in H^3(\O),\,\,u_0\in H_0^1(\O)\cap H^3(\O),\,\,\na\te_0\cdot n=0 \, \mbox{ on }\pa B_r.\ee
Then there exist  a small time $T_r>0$ and a unique classical solution $(\rho,u,\te)$ to the initial-boundary-value problem  (\ref{a1})--(\ref{bbb})    on $\O\times(0,T_r]$ satisfying
\be\la{mn6}
  \inf\limits_{(x,t)\in\O\times (0,T_r]}\n(x,t)\ge \frac{1}{2}\inf\limits_{x\in\O}\n_0(x),  \ee and
 \be\la{mn5}
 \begin{cases}
 ( \rho, u, \te) \in C([0,T_r];H^3(\O)),\quad \n_t\in C([0,T_r];H^2(\O)),\\
  (u_t,\te_t)\in C([0,T_r];H^1(\O)),\quad (u,\te)\in L^2(0,T_r;H^4(\O)).\end{cases}\ee
 \end{lemma}

As a direct consequence of Lemma \ref{w1},  the local  strong solutions to the Cauchy problem of approximation system \eqref{a1}  is established by the standard extension method, whose proof is postponed to the appendix.
\begin{lemma}\la{th3}
Under the conditions of Theorem \ref{th1} except $\te_0\geq 0$, assume further that $\te_0\in L^2(\r^2)$.
Then for any fixed $\de>0$, there exists a positive time  $T_\de$ such that the Cauchy problem (\ref{a1}) (\ref{q2}) (\ref{q3}) has a unique strong solution $(\n\geq 0,u,\te)$ on $\r^2\times(0,T_\de]$ satisfying 
\be\ba\la{bb}
\begin{cases}
\bar x^{a}\n\in L^\infty(0,T_\de;L^1(\r^2)\cap H^1(\r^2)\cap W^{1,q}(\r^2)),
\\\na u\in L^\infty(0,T_\de;H^1(\r^2))\cap L^2(0,T_\de;W^{1,q}(\r^2)),
\\\te\in L^\infty(0,T_\de;H^1(\r^2))\cap L^2(0,T_\de;H^2(\r^2)),
\\ \sqrt{\n}u,\,\sqrt{\n}\te,\,\sqrt{\n}u_t,\,\sqrt t\sqrt{\n}\te_t,\,\sqrt t\na^2\te\in L^\infty(0,T_\de;L^2(\r^2)),
\\ \na u_{t},\,\sqrt\n\te_t,\,\sqrt t\na\te_t\in L^2(\r^2\times(0,T_\de)),\quad\sqrt t\na^2\te\in L^2(0,T_\de;H^{1}(\r^2)).
\end{cases}
\ea\ee
\end{lemma}



 Next,  the following Poincar$\acute{e}$-type inequality can be found in \cite{L1}.
\begin{lemma}\la{w3}
For $m\in[2,\infty)$ and $\zeta\in(1+m/2,\infty)$, there exists a positive constant $C=C(m,\zeta)$ such that for either $\Omega=\r^2$ or $\Omega=\O$ with $r\geq 1$ and for any $v\in \tilde D^{1,2}(\Omega)\triangleq\{f\in H^1_{loc}|\na f\in L^2(\Omega)\}$,
\be\la{e2}\left(\int_\Omega\frac{|v|^m}{e+|x|^2}(\log(e+|x|^2))^{-\zeta}dx\right)^{1/m}\le C\|v\|_{L^2(B_1)}+C\|\na v\|_{L^2(\Omega)}.\ee
\end{lemma}

Now, we obtain the following key lemma as a result of Lemma \ref{w3}, which provides the weighted $L^p$-estimate for functions in $\tilde D^{1,2}$.
\begin{lemma}\la{w5}
Let $\bar x$ and $\eta_0$ be as in Theorem \ref{th1} and $\Omega$ as in Lemma \ref{w3}. Assume $\n\in L^\infty(\Omega)$ and $\int_{B_{N}}\n dx\geq M$ for some positive constants  $M$ and $1\le N \le r$. Then there exists a positive constant $C=C(N,M,\eta_0)$  such that for any $v\in \tilde D^{1,2}(\Omega)$,
\be\la{e3}\|\bar x^{-1}v\|_{L^2(\Omega)}\le C\|\sqrt\n v\|_{L^2(\Omega)}+C(\|\n\|_{L^\infty(\Omega)}^{1/2}+1)\|\na v\|_{L^2(\Omega)}.\ee
Moreover, for any $\epsilon >0$ and $0<\eta\le 1$, there exists a positive constant $C=C(N,M,\eta_0,\epsilon,\eta)$  such that for any $v\in \tilde D^{1,2}(\Omega)$,
\be\la{e5}\|\bar x^{-\eta}v\|_{L^{(2+\epsilon)/\eta}(\Omega)}\le C\|\sqrt\n v\|_{L^2(\Omega)}+C(\|\n\|_{L^\infty(\Omega)}^{1/2}+1)\|\na v\|_{L^2(\Omega)}.\ee
\end{lemma}
\begin{proof}
Denote $\bar v\triangleq \frac{1}{|B_{N}|}\int_{B_{N}} v dx$. By virtue of Poincar$\acute{e}$ inequality, one has
\bnn\ba\int_{B_{N}}\n dx |\bar v|^2
&\le 2\int_{B_{N}}\n|v|^2dx+2\int_{B_{N}}\n|v-\bar v|^2dx\\
&\le 2\|\sqrt\n v\|_{L^2(B_{N})}^2+2\|\n\|_{L^\infty(B_{N})}\|v-\bar v\|_{L^2(B_{N})}^2\\
&\le 2\|\sqrt\n v\|_{L^2(\Omega)}^2+C\|\n\|_{L^\infty(\Omega)}\|\na v\|_{L^2(\Omega)}^2,\\
\ea\enn
which leads to
\be\la{e4}\ba
\|v\|_{L^2(B_{N})}&\le\|v-\bar v\|_{L^2(B_{N})}+C|\bar v|\\
&\le C \|\na v\|_{L^2(B_{N})}+C\|\sqrt\n v\|_{L^2(\Omega)}+C\|\n\|_{L^\infty(\Omega)}^{1/2}\|\na v\|_{L^2(\Omega)}\\
&\le C\|\sqrt\n v\|_{L^2(\Omega)}+C(\|\n\|_{L^\infty(\Omega)}^{1/2}+1)\|\na v\|_{L^2(\Omega)}.
\ea\ee
Combining \eqref{e4} with  \eqref{e2} yields \eqref{e3} and \eqref{e5} directly. The proof of Lemma \ref{w5} is completed.
\end{proof}

Furthermore, in order to derive the integrability of velocity, the following  Caffarelli-Kohn-Nirenberg inequality \cite{ckn,cw} is needed.
\begin{lemma}\la{w4}
For any   $v\in C_0^\infty(\r^2)$, there exists a positive constant $C=C(k,l)$ such that
\be\notag\|v|x|^l\|_{L^{p}(\r^2)}\le C\|\na v|x|^k\|_{L^{2}(\r^2)},\ee
where $0<k<\infty$,  $k-1\le l< k$, and $p=2/(k-l)$.
\end{lemma}

Denote the material derivative of $f$,  effective viscous flux, and  vorticity respectively by
\be\la{hj1}\dot f\triangleq f_t+u\cdot\na f,\quad G\triangleq(2\mu+\lambda)\div u-P,\quad \curl u\triangleq\p_2 u_1-\p_1 u_2,\ee
then it follows from $(\ref{q1})$ that $G$ and $\curl u$ satisfy
 \be\la{h13}
\Delta G = \div (\rho\dot{u}),\quad\mu \Delta \curl u =\nabla^\perp\cdot(\rho\dot{u})\quad \text{in}\,\,\r^2, \ee
with $\na^\perp\triangleq (\p_2,-\p_1)$.
Applying the  standard $L^p$-estimate  to   elliptic equations (\ref{h13}), we thus get   the following elementary estimates (see \cite{lx}).

\begin{lemma} \la{w6}
Let $(\rho,u,\te)$ be a strong solution of (\ref{q1}) (\ref{q2})  (\ref{q3}).
Then for any $p\in [2,\infty),$ there exists a generic positive constant $C$ depending only on $\mu,$   $\lambda,$   $R$, and $p$ such that 
\be\la{h18}\|\na u \|_{L^p(\r^2)} \le C (\|G\|_{L^p(\r^2)}  +\|\curl u\|_{L^p(\r^2)}+ \|\n\te\|_{L^p(\r^2)}),\ee
\be\la{h19}\|\nabla G\|_{L^p(\r^2)} + \|\nabla \curl u\|_{L^p(\r^2)} \le C\|\rho\dot{u}\|_{L^p(\r^2)},\ee
\be\la{h20}\|G\|_{L^p(\r^2)}+\|\curl u\|_{L^p(\r^2)}\le C \|\rho\dot{u}\|_{L^2(\r^2)}^{(p-2)/p }\left(\|{\nabla u}\|_{L^2(\r^2)}
   +  \|\n\te\|_{L^2(\r^2)}\right)^{2/p}.\ee
\end{lemma}


Next, based on the Brezis-Wainger's inequality \cite{bg2}, we  arrive at the following improved logarithmic Sobolev inequality, referring  \cite{ot} for the detailed proof, which  gives the key $L^\infty_tL^2_x$-norm of $\na u$  in Section \ref{sec4}.
\begin{lemma}\la{w8}
For any $p\in(2,\infty)$, there exists a positive constant $C$ depending only on $p$ such that for $f\in \tilde D^{1,2}(\r^2)\cap W^{1,p}(\r^2),$   
\be\la{lo1}\|f\|_{L^\infty(\r^2)}^2\le C(1+\|\na f\|_{L^2(\r^2)}^2\ln (e+\|f\|_{W^{1,p}(\r^2)})).\ee
\end{lemma}
Finally, we state the  Beale-Kato-Majda type inequality (see \cite{bkm,k}), which will be used to estimate the gradient of the density. 
\begin{lemma}\label{lem-bkm}
For  any $p\in(2,\infty)$, assume that $\na u\in L^2(\r^2)\cap W^{1,p}(\r^2)$, then there is a positive constant  $C=C(p)$ such that  the following estimate holds
\bnn\ba
\|\na u\|_{L^\infty(\r^2)}\le& C\left(\|{\rm div}u\|_{L^\infty(\r^2)}+\|\curl u\|_{L^\infty(\r^2)} \right)\log(e+\|\na^2u\|_{L^p(\r^2)})\\
&+C\|\na u\|_{L^2(\r^2)} +C .
\ea\enn
\end{lemma}

\section{Proof of Theorem \ref{th1}}\la{sec3}
In this section, we mainly concentrate on the analysis in the whole space. Theorem \ref{th1} thus is proved by adopting the following strategy that based on the strong solutions to the Cauchy problem of approximation system with the damping term $\de\te$ obtained in Lemma \ref{th3}, we first  establish some necessary a priori estimates in Subsection \ref{sec3.2}, which are independent of $\de$, then take appropriate approximate limits by letting  $\de$ go to zero in Subsection \ref{sec3.3}. The limiting functions are the  desired strong solutions.

\subsection{Uniform a priori estimates}\la{sec3.2}
In this subsection, we establish all the necessary a priori estimates, which are uniform with respect to $\de$, for  the strong solutions $(\n^\de,u^\de,\te^\de)$ of the Cauchy problem (\ref{a1}) (\ref{q2}) (\ref{q3})  obtained in Lemma \ref{th3}.

First, for any fixed $\de>0$, let the initial data   $(\n_0^\de\geq 0,u_0^\de,\te_0^\de)$ satisfy \eqref{q4}, \eqref{co1}, with $(\n_0,u_0,\te_0)$ being replaced by  $(\n_0^\de,u_0^\de,\te_0^\de)$, 
and $\te_0^\de\in L^2$.
Without loss of generality, we assume that there exists a positive constant $N_1\geq 1$ such that
\be\la{oh}\int_{B_{N_1}}\n_0^\de dx\geq 1/2.\ee
Then Lemma \ref{th3} implies the Cauchy problem (\ref{a1}) (\ref{q2}) (\ref{q3}) admits a unique strong solution $(\n^\de,u^\de,\te^\de)$ satisfying \eqref{bb} on $\r^2\times(0,T_\de]$ for some $T_\de>0$.   For the  simplicity of notations, we omit the superscript $\de$ in the remaining subsection. 

For $q,\,a,\,\bar x$, and $\eta_0$ as in Theorem \ref{th1}, define
\be\ba\la{bb0}
\psi(t)\triangleq \psi_1(t)+\psi_2(t)+\psi_3(t)+\sup_{0\le s\le t}\|\na(\bar{x}^{a}\n)\|_{L^{q}},
\ea\ee
with
\be\la{b0}
\psi_1(t)\triangleq
1+\sup_{0\le s\le t}\left(\|\sqrt{\n} u\|_{L^2}^2+\|\na u\|_{L^2}^2+\|\sqrt{\n}\te\|_{L^2}^2\right),
\ee
\be\la{bq0}
\psi_2(t)\triangleq 1+\sup_{0\le s\le t}\left(\|\bar{x}^{a}\n\|_{L^{1}\cap L^{\infty}}+\|\na(\bar{x}^{a}\n)\|_{L^{2}}\right),
\ee
\be\la{bp0}
\psi_3(t)\triangleq 1+\sup_{0\le s\le t}\left(\|\sqrt{\n}u_t\|_{L^2}^2+\|\na\te\|_{L^2}^2+\de\|\te\|_{L^2}^2 \right)
.
\ee

Now we state the following crucial a priori estimate on $\psi$.
\begin{pro}\la{pr2}
For any fixed $\de>0$, let the initial data $(\n_0\geq 0,u_0,\te_0)$ satisfy (\ref{q4})
, (\ref{co1}), (\ref{oh}), and $\te_0\in L^2$, assume that $(\n,u,\te)$ is the strong solution to the Cauchy problem (\ref{a1}) (\ref{q2}) (\ref{q3}) on $\r^2\times(0,T_\de]$ obtained in Lemma \ref{th3}. We define
\be\la{cx} T_\de^\ast= \sup\left\{T\le T_\de\left|\sup_{0\le t\le T}\|\sqrt\n u\|_{L^2}^2\le M_1\right.\right\},\ee
for some positive constant $M_1$ determined later, which is  depending only on $\mu$,  $\lambda$, $R,$ $c_\upsilon,\,\ka,\,q,\,a,\,\eta_0,$ $N_1$, and  $\psi_{0}$.
 Then there exist   positive constants  $T^\ast$ and $\tilde M_1$ depending only on $\mu$,  $\lambda$, $R,$ $c_\upsilon,\,\ka,\,q,\,a,\,\eta_0,$ $N_1$, and  $\psi_{0}$  
 but independent of $\de$ such that
\be\ba\la{d01}
\psi(T^\ast)\le M_1/2,
\ea\ee
and
\be\ba\la{d03}
&\sup_{0\le t\le T^\ast}(\|\na^2 u\|_{L^2}+\|\na u\bar x^{\al/2}\|_{L^{2}}+\|u\|_{L^{q_1}})\\
&+\int_0^{T^\ast}\left(\|\na u_t\|_{L^2}^2+\|\sqrt\n\te_t\|_{L^2}^2+\|\na^2\te\|_{L^2}^2+\|\na^2 u\|_{L^{q}}^2\right)dt\le \tilde M_1,
\ea\ee
where $\al\triangleq\min\{{a/2},1\}$,  $q_1=4/\al$, and
\be\la{000}\ba
\psi_0\triangleq&1+\|\bar{x}^{a}\n_0\|_{L^{1}\cap H^{1}\cap W^{1,q}}+\|\sqrt{\n_0}u_0\|_{L^2}^2+\|\sqrt{\n_0}\te_0\|_{L^2}^2\\
       &+\|\na u_0\|_{H^1}^2+\de\|\te_0\|_{L^2}^2+\|\na\te_0\|_{L^2}^2+\|g\|_{L^2}^2.\ea\ee
\end{pro}

In the following subsection, let   $A$, $C$, and $C_1$ denote some generic constants depending only on $\mu,\,\lambda,\,R,\,c_\upsilon,\,\ka$, $q$, $a$, $\eta_0$, $N_1$, and $\psi_0$, but independent of  $\de$, and we write $C(\eta)$ to emphasize that the constant  depends on $\eta$ as well.

In order to prove Proposition \ref{pr2}, next our  goal is to establish  the following weighted $L^2$-norm on $\na u$, which is the key  to controlling the most difficult term $\frac{d}{dt}\int F(\na u)\te dx$ while estimating $\psi_3$ and further obtaining all the uniform estimates with respect to $\de$. The method  here is motivated by the techniques in the theory of singular integrals.

\begin{lemma}\la{lee0}
Under the conditions of Proposition \ref{pr2}, there exist  positive constants  $C$  and $T_1=T_1(  M_1,N_1,\|\n_0\|_{L^1})$ independent of $\de$ such that for any $0<T\le T_1$,
\be\la{hii}\ba
\|\na u\bar x^{\al/2}\|_{L^2}+\|u\|_{L^{q_1}}\le C\psi_1\psi_2^{3/2}\psi_3^{1/2},
\ea\ee
with  $ \al\triangleq\min\{{a/2},1\}$ and $ q_1=4/\al.$  Moreover,
\be\la{term}
\left|\int F(\na u)\te dx\right|
\le C\psi_1^2\psi_2^2\psi_3^{7/8}.
\ee
\end{lemma}
\begin{proof}
First, we give some basic weighted estimates. Note that $\eqref{a1}_1$ implies that 
\be\notag\|\n\|_{L^1}(t)=\|\n_0\|_{L^1},\quad t\in(0,T_\de],\ee
which along with \eqref{oh} and  \eqref{cx} leads to
\be\la{t2}\inf_{0\le t\le T_1}\int_{B_{2N_1}}\n dx\geq \frac{1}{4},\ee
for some positive constant $T_1$ depending only on $M_1$, $N_1$, and $\|\n_0\|_{L^1}$ but independent of $\de$, referring  \cite[Lemma 3.2]{liliang} for the detailed proof. In fact, here we require $T_\de^\ast\geq T_1$ for any $\de>0$, otherwise the standard extension method (see \cite{H-L,hulx}) can ensures it holds.
From now on,   we will  always assume that $T\le T_1.$
Thus, the combination of \eqref{t2} and Lemma \ref{w5} yields that for any $\ve>0$ and $\eta\in(0,1],$ every $v\in\ti D^{1,2}(\r^2)$ satisfies
\be \la{bb1}\ba
\|v\bar x^{-1}\|_{L^2}^2+\|v\bar x^{-\eta}\|_{L^{(2+\ve)/\eta} }^2
\le C(\epsilon,\eta) \|\sqrt{\n}v\|_{L^2}^2 +C(\epsilon,\eta) (\|\n\|_{L^\infty}+1) \|\na v\|_{L^2 }^2.
\ea\ee
In particular, one has
\be\la{uu}\ba
\|u\bar x^{-\eta}\|_{L^{(2+\ve)/\eta}}\le C(\epsilon,\eta)\psi_1^{1/2}\psi_2^{1/2},\quad \|\n^{\eta}u\|_{L^{(2+\ve)/\eta}}\le C(\epsilon,\eta)\psi_1^{1/2}\psi_2^{1/2+\eta},\ea\ee
\be\la{uu1}\ba
\|\te\bar x^{-\eta}\|_{L^{(2+\ve)/\eta}}
\le C(\epsilon,\eta)(\psi_1^{{1/2}}+\psi_2^{1/2}\|\na\te\|_{L^2}),\quad \|\n^{\eta}\te\|_{L^{(2+\ve)/\eta}}
\le C(\epsilon,\eta)\psi^{1+\eta}.\ea\ee

Next, we show the proof of  \eqref{hii}. One deduces from \eqref{a1}$_2$ that
$$(2\mu+\lambda)\div u=-(-\Delta)^{-1}\div(\n\dot u)+P,$$
which gives
\be\ba\la{good}
\mu \na u&=(\mu+\lambda)\na(-\Delta)^{-1}\na\div u-\na(-\Delta)^{-1}(\n\dot u)-\na(-\Delta)^{-1}\na P\\
         &=-\frac{\mu+\lambda}{2\mu+\lambda}\na(-\Delta)^{-1}\na((-\Delta)^{-1}\div (\n\dot u)-P)\\
         &\quad-\na(-\Delta)^{-1}(\n\dot u)-\na(-\Delta)^{-1}\na P.
\ea\ee
In view of the theory of singular integrals, we can  rewrite \eqref{good} as the sum of the ``good'' part and the ``bad'' part:
\be\la{key3}\p_i u\triangleq \mathcal{T}_i P+K_i\ast (\n\dot u),\quad i=1,2,\ee
 where in ``good'' part, $\mathcal{T}_i=-\frac{1}{2\mu+\lambda}\p_i(-\Delta)^{-1}\na\,(i=1,2)$ is the singular integral operator of convolution type with Calderon-Zygmund kernel  (refer to \cite{stein1}) which is bounded in $L^p\,(1<p<\infty)$ and in ``bad'' part, the  kernel $K_i(x)$ satisfies
\be\la{ker}|K_i(x)|\le C|x|^{-1},\quad|\na K_i(x)|\le C|x|^{-2},\quad i=1,2.\ee

For the ``good'' part, we adopt the well-known theory of  $A_p$ weights to estimate. Precisely, since for $-1<k<1$, $|x|^{2k}$ belongs to the class $A_2$ \cite[p.194]{stein},  let $\tilde\al \in(0,\min\{a/2,1\}]$, it holds that for any $0<\beta<\tilde\al$ and $i=1,2$,
\be\la{key1}\|\mathcal{T}_i P|x|^{\beta}\|_{L^2}\le C(\beta)\|P|x|^{\beta}\|_{L^2}\le C(\beta)\|\sqrt\n\te\|_{L^2}\|\n\bar x^a\|_{L^\infty}^{1/2}\le C(\beta)\psi_1^{1/2}\psi_2^{1/2}.\ee

Then, we claim that $\n\dot u$ has  the following cancellation condition: for any $t\in (0,T_1]$,
\be\la{zer}\int (\n\dot u) (y,t)dy=0,\ee
which combined with the property \eqref{ker} of kernel $K_i$  implies the ``bad'' part satisfies that  for any $x\neq0$, $t\in (0,T_1]$, $i=1,2$,
\be\la{k1}\ba
&|K_i\ast (\n\dot u)(x,t)| \\
&=\left|\int (K_i(x-y)-K_i(x))(\n\dot u)(y,t)dy\right|\\
&\le\int_{|y|\le {|x|/2}} |K_i(x-y)-K_i(x)||\n\dot u|(y,t)dy\\
&\quad+\int_{|y|>{|x|/2}}|K_i(x-y)-K_i(x)||\n\dot u|(y,t)dy\\
&\le \int_{|y|\le {|x|/2}}\left|\int_0^1 y\cdot\na K_i(x-\eta y)d\eta\right||\n\dot u|(y,t)dy\\
&\quad+\int_{|y|> {|x|/2}}|K_i(x-y)||\n\dot u|(y,t)dy+|K_i(x)|\int_{|y|> {|x|/2}}|\n\dot u|(y,t)dy\\
&\le C\int_{|y|\le {|x|/2}}\int_0^1 |y||x-\eta y|^{-2}|\n\dot u|(y,t)d\eta dy\\
&\quad+ C\int_{|y|>{|x|/2}}|x-y|^{-1}|\n\dot u|(y,t)dy+C|x|^{-1}\int_{|y|> {|x|/2}}|\n\dot u|(y,t)dy\\
&\le C|x|^{-2}\int_{|y|\le {|x|/2}}|y||\n\dot u|(y,t)dy+C|x|^{-1-\tilde\alpha}\int_{|y|> {|x|/2}}|y|^{\tilde\al}|\n\dot u|(y,t)dy\\
&\quad+ C|x|^{-\tilde\alpha}\int_{|y|> {|x|/2}}|x-y|^{-1}|y|^{\tilde\alpha}|\n\dot u|(y,t)dy\\
&\le C|x|^{-1-\tilde\alpha}\||y|^{\tilde\alpha} \n\dot u\|_{L^1}
+ C|x|^{-\tilde\alpha}\mathcal{I}_1(|y|^{\tilde\alpha}|\n\dot u|),
\ea\ee
where $\mathcal{I}_1$ is the Riesz potential of order 1 and $\tilde\al \in(0,\min\{a/2,1\}]$.
Thus, for any $0<\beta<\tilde\al$,
\be\ba\la{key2}
\|K_i\ast (\n\dot u)|x|^{\beta}\|_{L^2(\r^2\setminus B_1)}
&\le C\||x|^{\beta-\tilde\alpha}\|_{L^{4/(\tilde\alpha-\beta)}(\r^2\setminus B_1)}\|\mathcal{I}_1(|y|^{\tilde\alpha}|\n\dot u|)\|_{L^{4/(2-\tilde\alpha+\beta)}}\\
&\quad+C\||x|^{\beta-1-\tilde\al}\|_{L^2(\r^2\setminus B_1)}\||y|^{\tilde\alpha}\n\dot u\|_{L^1}\\
&\le C(\tilde\al,\beta)(\||y|^{\tilde\alpha}\n\dot u\|_{L^{4/(4-\tilde\alpha+\beta)}}+\||y|^{\tilde\alpha}\n\dot u\|_{L^1})\\
&\le C(\tilde\al,\beta)\psi_1\psi_2^{3/2}\psi_3^{1/2},
\ea\ee
where we have used Hardy-Littlewood-Sobolev inequality and  the following fact: for any $p\in[1,4/3)$ and $\eta\in[0,a/2]$,
\be\ba\notag
\||y|^\eta \n\dot u\|_{L^p}
&\le \||y|^\eta\n u_t\|_{L^p}+\||y|^\eta\n u\cdot\na u\|_{L^p}\\
&\le \|\sqrt\n u_t\|_{L^2}\|\n\bar y^{2\eta}\|_{L^{p/(2-p)}}^{1/2}+\|\na u\|_{L^2}\|\n\bar y^a\|_{L^{4p/(4-3p)}}\|u\bar y^{\eta-a}\|_{L^4}\\
&\le C(\eta,p)\psi_1\psi_2^{3/2}\psi_3^{1/2},
\ea\ee
due to \eqref{uu}.
Taking \eqref{key1} and \eqref{key2} into \eqref{key3}, we conclude that for $i=1,2$,
\be\ba\la{fin}
\|\p_i u|x|^{\beta}\|_{L^2}
&\le \|\na u|x|^{\beta}\|_{L^2(B_1)}+\|K_i\ast (\n\dot u)|x|^{\beta}\|_{L^2(\r^2\setminus B_1)}+\|\mathcal{T}_i P|x|^{\beta}\|_{L^2(\r^2\setminus B_1)}\\
&\le C(\tilde\al,\beta)\psi_1\psi_2^{3/2}\psi_3^{1/2}.
\ea\ee

In particular, choosing  $\beta=\al/2,3\al/4$ in \eqref{fin} with  $\al\triangleq \min\{a/2,1\}$ and applying  Lemma \ref{w4},  one derives that for $q_1=4/\al$,
\be\notag\ba
\|u\|_{L^{q_1}}+\|\na u\bar x^{\al/2}\|_{L^2}
&\le C\|\na u|x|^{\al/2}\|_{L^{2}}+C(\|\na u\|_{L^2}+\|\na u|x|^{3\al/4}\|_{L^2})\\
&\le C\psi_1\psi_2^{3/2}\psi_3^{1/2}.
\ea\ee

It remains to prove the claim \eqref{zer}. Indeed, taking text function $\varphi_r$ as follows:
\be\la{we}0\le\varphi_r\in C_0^\infty(B_{r}),\,\,\varphi_r\equiv1\,\text{in}\,B_{r/2},\,\,|\na^i\varphi_r|\le Cr^{-i},\,i=1,2,\ee
we deduce from \eqref{a1}$_2$ that
\be\ba\notag
\left|\int (\n\dot u) \varphi_r dy\right|
=&\left|\mu\int\Delta u\varphi_r dy+(\mu+\lambda)\int\na\div u \varphi_r dy-\int\na P\varphi_r dy\right|\\
\le&C\int|\na u||\na\varphi_r|dy+C\int\n|\te||\na\varphi_r|dy\\
\le&C\|\na u\|_{L^2(r/2\le|x|\le r)}+C\|\n\te\|_{L^2(r/2\le|x|\le r)}.
\ea\ee
Letting $r\rightarrow 0$ and using Lebesgue Dominated Convergence Theorem and \eqref{bb}, we arrive at \eqref{zer} directly. 

Finally, noticing that  the Sobolev inequality combined with the standard elliptic estimates to \eqref{a1}$_2$, \eqref{uu}, and \eqref{uu1} shows that
\be\ba\la{nau}
\|\na u\|_{L^6}
&\le C(\|\na^2 u\|_{L^{3/2}}+\|\na u\|_{L^2})\\
&\le C( \|\n  u_t\|_{L^{3/2}}+\|\n  u\cdot\na u\|_{L^{3/2}}+ \|\n\na\te\|_{L^{3/2}}+\|\na\n\te \|_{L^{3/2}}+\|\na u\|_{L^2})\\
&\le C\|\n\|_{L^3}^{1/2}\|\sqrt\n  u_t\|_{L^2}+C\|\n u\|_{L^6}\|\na u\|_{L^2}+C\|\n\|_{L^6} \|\na \te\|_{L^2}\\
&\quad+C\|\na \n\bar x\|_{L^2}\|\te\bar x^{-1}\|_{L^6}+C\|\na u\|_{L^2}\\
&\le C\psi_1\psi_2^{3/2}\psi_3^{1/2}.
\ea\ee
This as well as  \eqref{uu1}  and \eqref{hii} leads to
\be\notag\ba
\left|\int F(\na u)\te dx\right|
&\le \|\na u\bar x^{\al/2}\|_{L^2}^{1/2}\|\te\bar x^{-{\al/4}}\|_{L^{12/\al}}\|\na u\|_{L^2}^{(6-\al)/4}\|\na u\|_{L^6}^{\al/4}\\
&\le C(\psi_1\psi_2^{3/2}\psi_3^{1/2})^{1/2}(\psi_1^{{1/2}}+\psi_2^{1/2}\psi_3^{1/2})\psi_1^{(6-\al)/8} (\psi_1\psi_2^{3/2}\psi_3^{1/2})^{\al/4}\\
&\le C\psi_1^2\psi_2^2\psi_3^{7/8},
\ea\ee
due to $\al\le 1$. The proof of Lemma \ref{lee0} is completed.
\end{proof}

\begin{remark}
It is worth noting that the cancellation condition (\ref{zer}) is an exclusive property for two-dimensional Cauchy problem, which plays an  important role in the whole analysis. Indeed,  it cancels out the singularity of the ``bad'' part in (\ref{key3}) at far field, thereby ensures the ``bad'' part  has higher decay rate (see (\ref{k1})).
\end{remark}

\begin{lemma}\la{le01}
Let $T_1$ be as in Lemma \ref{lee0}. Then there exist  positive constants  $C$ and $A\geq 1$ such that for any $0<T\le T_1$,
\be\la{b155}\psi_1(T)\le  C+C\int_0^T\psi^Adt.\ee
\end{lemma}
\begin{proof}
First, multiplying \eqref{a1}$_2$  by $u$  and integrating the resulting equality by parts yield
\be\ba\la{1}
&\frac{1}{2}\frac{d}{dt}\int\n |u|^2+\int\left(\mu|\na u|^2+(\mu+\lambda)(\div u)^2\right)dx\\
&\le C\|\n\|_{L^\infty}^{1/2}\|\sqrt\n \te\|_{L^2}\|\na u\|_{L^2}\\
&\le C\psi^{3/2}.
\ea\ee

Next, by virtue of  the  standard $L^2$-estimate to \eqref{a1}$_2$, \eqref{uu}, and \eqref{uu1}, one gets
\be \ba\la{b04}
2\|\na u\|_{H^1}
&\le C \left(  \|\n  u_t\|_{L^2}+\|\n  u\cdot\na u\|_{L^2}+ \|\na P\|_{L^2}+\|\na u\|_{L^2}\right)\\
&\le C \|\n\|_{L^\infty}^{1/2}\|\sqrt\n u_t\|_{L^2}
   +C\|\n u\|_{L^4}\|\na u\|_{L^2}^{1/2}\|\na u\|_{H^1}^{1/2}\\
&\quad+C\|\n\|_{L^\infty}\|\na\te\|_{L^2}+C\|\na\n\bar x^a\|_{L^q}\|\te\bar x^{-a}\|_{L^{2q/(q-2)}}+C\|\na u\|_{L^2}\\
&\le \|\na u\|_{H^1}+C\psi^A,
\ea\ee
which as well as \eqref{uu} shows that for any $\eta\in(0,1]$,
\be\ba\la{b111}
\|u\bar x^{-\eta}\|_{L^\infty}
&\le C(\eta)(\|u\bar x^{-\eta}\|_{L^{4/\eta}}+\|\na(u\bar x^{-\eta})\|_{L^3})\\
&\le C(\eta)(\|u\bar x^{-\eta}\|_{L^{4/\eta}}+\|\na u\|_{L^2}^{2/3}\|\na u\|_{H^1}^{1/3}+\|u\bar x^{-1}\|_{L^{3}}\|\bar x^{-\eta}\na\bar x\|_{L^{\infty}})\\
&\le C(\eta)\psi^A.
\ea\ee

Now, multiplying \eqref{a1}$_2$ by $2u_t$ and integrating by parts, we derive
\be\la{b5}\ba
&\frac{d}{dt}\int (\mu|\na u|^2+(\mu+\lambda)(\div u)^2-2P\div u)dx+\int \n|u_t|^2dx\\
&\le -2\int P_t\div u dx+\int \n|u|^2|\na u|^2dx\\
&\le C\int(|P||u|+|\na\te|)|\na^2 u|dx+C\int(|\na u|^3+\de|\te||\na u|+P|\na u|^2)dx\\
&\quad+C\|\n\bar x^{a}\|_{L^\infty}\|u\bar x^{-1/2} \|_{L^\infty}^{2}\|\na u\|_{L^2}^{2}\\
&\le C(\|\n\bar x^a\|_{L^\infty}\|\te\bar x^{-a/2}\|_{L^4}\|u\bar x^{-a/2}\|_{L^4}+\|\na\te\|_{L^2})\|\na^2 u\|_{L^2}+C\de\|\te\|_{L^2}\|\na u\|_{L^2}\\
&\quad+C(\|\na u\|_{L^2}+\|\n\|_{L^\infty}^{1/2}\|\sqrt\n\te\|_{L^2})\|\na u\|_{L^2}\|\na u\|_{H^1}+C\psi^A\\
&\le C\psi^A,
\ea\ee
where we have used  \eqref{uu}, \eqref{uu1}, \eqref{b04}, \eqref{b111}, and the following fact:
\be\la{pf}P_t+\div(Pu)=c_\nu^{-1}R(F(\na u)+\ka\Delta \te-\de\te-P\div u).\ee


Next, multiplying \eqref{a1}$_3$ by $\te $, one deduces from  integration by parts, \eqref{term}, and \eqref{uu1} that
\be\ba\la{on}
\frac{c_\upsilon}{2}\frac{d}{dt}\int\n\te^2dx+\int(\ka|\na\te|^2+\de|\te|^2)dx
&=\int F(\na u)\te dx- R\int \n\te^2\div u dx\\
&\le C\psi^A+C\|\n\bar x\|_{L^{6}}\|\te\bar x^{-{1/2}}\|_{L^6}^2\|\na u\|_{L^2}\\
&\le C\psi^A.
\ea\ee

Finally, adding \eqref{1} and \eqref{b5} to \eqref{on} and  
integrating the resulting inequality with respect to $t$, we 
obtain that
\be\ba\notag
&\psi_1(T)\le C+C\int_0^T\psi^A dt+C\sup_{0\le t\le T}\|P\|_{L^2}^2,
\ea\ee
which combined with the following fact:
\be\ba\notag
\frac{d}{dt}\|P\|_{L^2}^2
&\le C\int(|P|^2|\na u|+|P||\na u|^2+\de|\te||P|+|\na\te||\na P|)dx\\
&\le C\|\n\bar x^a\|_{L^\infty}^2\|\te\bar x^{-a}\|_{L^4}^2\|\na u\|_{L^2}+C\|\n\|_{L^\infty}^{1/2}\|\sqrt\n\te\|_{L^2}\|\na u\|_{L^2}\|\na u\|_{H^1}\\
&\quad+C\|\sqrt\n\te\|_{L^2}^2+C(\|\n\|_{L^\infty}\|\na\te\|_{L^2}^2+\|\na\n\bar x\|_{L^q}\|\te\bar x^{-1}\|_{L^{2q/(q-2)}}\|\na\te\|_{L^2})\\
&\le C\psi^A,
\ea\ee
due to \eqref{pf}, \eqref{uu1}, and \eqref{b04}, 
infers \eqref{b155}. The proof of Lemma \ref{le01} is finished.
\end{proof}


\begin{lemma}\la{le3}
Let  $q$ be as in Theorem \ref{th1} and $T_1$ as in Lemma \ref{lee0}. Then for any $p'>2$, there exist  positive constants $C$, $C(p')$, and $A\geq 1$ such that for any $0<T\le T_1$,
\be\la{b14}\psi_2(T)\le C+C(p')\psi_3^{q/(2q+p'q-2p')}(T)\left(1+\int_0^T\psi^Adt\right).\ee
\end{lemma}
\begin{proof}
First,  combining the standard $L^2$-estimate to \eqref{a1}$_3$ with  \eqref{uu1}, \eqref{b111}, and \eqref{b04} 
 leads to
\be\ba\la{b05}
\|\na\te\|_{H^1}
&\le C(\|\n (\te_t+ u \cdot \na \te) \|_{L^2} +\|\na u\|_{L^4}^2+\|P\div u\|_{L^2}+\de\|\te\|_{L^2}+\|\na \te\|_{L^2})\\
& \le C\|\n\te_t\|_{L^2}+C\|\n\bar x\|_{L^\infty}\|u\bar x^{-1}\|_{L^\infty}\|\na\te\|_{L^2}+C\|\na u\|_{L^2}\|\na u\|_{H^1}\\
&\quad+C\|\n\te\|_{L^4}\|\na u\|_{L^2}^{1/2}\|\na u\|_{H^1}^{1/2}+C\de\|\te\|_{L^2}+C\|\na \te\|_{L^2}\\
& \le  C\psi^A(\|\sqrt\n\te_t\|_{L^2}+1).\ea\ee

Next, applying the standard $L^{q}$-estimate to \eqref{a1}$_2$, one gets
\be \ba\la{r2}
\|\na u\|_{W^{1,q}}
&\le C\left(\|\n u_t\|_{L^{q}}+\|\n  u\cdot\na u\|_{L^{q}}+\|\na P\|_{L^{q}}+\|\na u\|_{L^q}\right)\\
&\le C\|\n\bar x\|_{L^\infty}\|u_t\bar x^{-1}\|_{L^{q}}+C\|\n u\|_{L^{2q}}\|\na u\|_{L^2}^{1/q}\|\na u\|_{H^1}^{1-1/q}\\
&\quad+ C\|\n\|_{L^\infty}\|\na\te\|_{L^2}^{2/q}\|\na \te\|_{H^1}^{1-2/q}+C\|\bar x\na\n\|_{L^q}\|\te\bar x^{-1}\|_{L^\infty} +C\|\na u\|_{H^1}\\
&\le C\psi^A(\|\na u_t\|_{L^2}+\|\sqrt\n\te_t\|_{L^2}+1),
\ea\ee
where  we have used  \eqref{bb1}, \eqref{uu}, \eqref{b04}, \eqref{b05}, and the following fact: for any $\eta\in(0,1]$,
\be\ba\la{b112}
\|\te\bar x^{-\eta}\|_{L^\infty}
&\le C(\eta)(\|\te\bar x^{-\eta}\|_{L^{4/\eta}}+\|\na(\te\bar x^{-\eta})\|_{L^q})\\
&\le C(\eta)(\|\te\bar x^{-\eta}\|_{L^{4/\eta}}+\|\na \te\|_{L^2}^{2/q}\|\na\te\|_{H^1}^{1-2/q}+\|\te\bar x^{-1}\|_{L^{q}}\|\bar x^{-\eta}\na\bar x\|_{L^\infty})\\
&\le C(\eta)\psi^A(\|\sqrt\n\te_t\|_{L^2}+1),
\ea\ee
owing to \eqref{uu1} and \eqref{b05}.

Now, denote $\omega=\bar x^a\n$. By virtue of $\eqref{a1}_1$, it holds that
\be\la{b9}\p_t\omega+\div(\omega u)-a\omega u\cdot\na \log \bar x=0.\ee

On the one hand, for any $p\in[1,\infty)$, multiplying \eqref{b9} by $p\omega^{p-1}$  and integrating the resulting equality over $\r^2,$  one obtains  that
\be\la{b11}\ba
\frac{d}{dt}\|\omega\|_{L^p}  &\le C(\|\div u\|_{L^\infty}+\|u\cdot\na \log \bar x\|_{L^\infty})\| \omega\|_{L^p},
\ea\ee
where the constant $C$ is independent of $p$.

On the other hand,  applying the operator $\p_i (i=1,2)$ to \eqref{b9} 
along with some standard calculation yields that
 \be\la{b50}\ba
\frac{d}{dt}\|\na \omega\|_{L^{2}}
&\le C(\|\na u\|_{L^\infty}+\|u\cdot\na \log \bar x\|_{L^\infty})\|\na \omega\|_{L^{2}}\\
&\quad+C\|\omega\|_{L^\infty}(\|\na^2u \|_{L^{2}}+\||\na u||\na\log\bar x|\|_{L^{2}}+\||u||\na^2\log \bar x|\|_{L^{2}}).
\ea\ee

Hence,   we deduce from \eqref{b11}, \eqref{b50}, \eqref{uu}, \eqref{b04}, \eqref{b111},  \eqref{r2}, and Gagliardo-Nirenberg inequality that for    any $p'>2$,
\be\notag\ba
\sup_{0\le t\le T}\|\omega\|_{L^p\cap H^1}
&\le C+C(p')\int_0^T\psi^A(\|\na u\|_{L^{p'}}^{p'(q-2)/(2q+p'q-2p')}\|\na u\|_{W^{1,q}}^{2q/(2q+p'q-2p')}\!+\!1)dt\\
&\le C+C(p')\int_0^T\psi^A(\|\na u_t\|_{L^2}+\|\sqrt\n\te_t\|_{L^2}+1)^{2q/(2q+p'q-2p')}dt\\
&\le C+C(p')\psi_3^{q/(2q+p'q-2p')}(T)\left(1+\int_0^T\psi^Adt\right),
\ea\ee
which implies \eqref{b14}, because of the  uniformity with respect to $p$. The proof of Lemma \ref{le3} is finished.
\end{proof}

With   Lemmas \ref{lee0}--\ref{le3} at hand, we can estimate the most troublesome term $\psi_3$.
\begin{lemma}\la{le4}
Let $T_1$ be as in Lemma \ref{lee0}. Then there exist positive  constants $C$ and $A\geq1$ such that for any $0<T\le T_1$,
\be\ba\la{b47}
\psi_3(T)+\int_0^T(\|\na u_t\|_{L^2}^2+\|\sqrt\n\te_t\|_{L^2}^2)dt\le C+C\int_0^T\psi^Adt.
\ea\ee
\end{lemma}
\begin{proof}
We will prove this lemma in four steps.

{\it Step 1}
Differentiating \eqref{a1}$_2$ with respect to $t$ gives
\be\la{nt0}\ba
&\n (u_{tt}+u\cdot\na u_t)-\mu\Delta u_t-(\mu+\lambda) \na\div  u_t\\
&=-\n_t (u_t+u\cdot\na u)-\n u_t\cdot\na u -\na P_t.\\
\ea\ee
Then multiplying \eqref{nt0} by $u_t$ and integrating the resulting equality over $\r^2$, it holds that
\be\la{b26} \ba
&\frac{1}{2}\frac{d}{dt}\int\rho|u_t|^2dx +\int\left(\mu|\na u_t|^2+(\mu+\lambda)(\div u_t)^2\right) dx \\
&=-2\int\n u\cdot \na u_t\cdot u_t dx -\int\n u \cdot\na (u\cdot\na u\cdot u_t)dx\\
&\quad-\int \n u_t\cdot\na u\cdot u_t dx+\int P_t\div u_t dx\triangleq \sum_{i=1}^4 I_i.
\ea \ee
It thus follows from   \eqref{bb1}--\eqref{uu1}, \eqref{b04}, \eqref{b111}, 
 and \eqref{a1}$_1$ that
\be\ba\la{b27}
|I_1|&\le C\|\n\bar x\|_{L^\infty}^{1/2}\|u\bar x^{-{1/2}}\|_{L^\infty}\|\sqrt\n u_t\|_{L^2}\|\na u_t\|_{L^2}\\
   &\le C\psi^A\|\na u_t\|_{L^2}, 
\ea\ee
\be\ba\la{b28}
|I_2|&\le\int\n|u|(|\na u|^2|u_t|+|u||\na^2 u||u_t|+|u||\na u||\na u_t|)dx\\
   &\le C\|\sqrt\n u_t\|_{L^2}\|\n\bar x\|_{L^\infty}^{1/2}(\|u\bar x^{-\!{1/2}\!}\|_{L^\infty}\|\na u\|_{L^2}\|\na u\|_{H^1}+\|u\bar x^{-\!{1/4}\!}\|_{L^{\infty}}^2\|\na^2 u\|_{L^2})\\
   &\quad+ C\|\n\bar x\|_{L^\infty}\|u\bar x^{-{1/2}}\|_{L^\infty}^2\|\na u\|_{L^2}\|\na u_t\|_{L^2}\\
   &\le C\psi^A(\|\na u_t\|_{L^2}+1), 
\ea\ee
\be\ba\la{b29}
 |I_3|&\le C\|\n\bar x\|_{L^\infty}^{1/2}\|\sqrt\n u_t\|_{L^2}\|u_t\bar x^{-{1/2}}\|_{L^6}\|\na u\|_{L^2}^{2/3}\|\na u\|_{H^1}^{1/3}\\
      &\le C\psi^A(\|\na u_t\|_{L^2}+1),
 \ea\ee
 and
\begin{equation}\ba \la{b30}
|I_4|
\le & C\| \na u_t  \|_{L^2}(\|\rho \theta_t\|_{L^2} +\|\n_t\te\|_{L^2})\\
\le & C\| \na u_t  \|_{L^2}\|\rho \theta_t\|_{L^2}+C\| \na u_t  \|_{L^2}\|\n\te\|_{L^4}\|\na u\|_{L^2}^{1/2}\|\na u\|_{H^1}^{1/2} \\
&+C\| \na u_t  \|_{L^2}\|\na\n\bar x\|_{L^q}\|u\bar x^{-{1/2}}\|_{L^{4q/(q-2)}}\|\te\bar x^{-{1/2}}\|_{L^{4q/(q-2)}}\\
\le & C\| \na u_t  \|_{L^2}(\|\sqrt\n \te_t\|_{L^2}+\|\n^{3/2} \te_t\|_{L^2}+\psi^A). 
\ea \end{equation}

Putting \eqref{b27}--\eqref{b30} into \eqref{b26} and using Cauchy's inequality, one has
\be \ba\la{b32}
\frac{d}{dt}\int \rho|u_t|^2dx +\mu\int|\na u_t|^2 dx \le \tilde C(\|\sqrt\n \te_t\|_{L^2}^2+\|\n^{3/2} \te_t\|_{L^2}^2)+C\psi^A.
\ea \ee

{\it Step 2}
Multiplying \eqref{a1}$_3$ by $2\te_t$ and integrating the resulting equality  lead to
\be\ba\la{b33}
&\frac{d}{dt}\int(\ka|\na\te|^2+\de|\te|^2-2 F(\na u)\te)dx+ 2c_\upsilon\int\rho|\te_t|^2dx\\
&=-2c_\upsilon\int\n \te_t u\cdot\na\te dx-2R\int\n\te\te_t\div u dx-2\int F(\na u)_t\te dx\triangleq  \sum_{i=1}^3J_i.
\ea\ee
We deduce from  \eqref{uu1}, \eqref{b04},  
\eqref{b111}, and \eqref{hii}  that
\be\la{r18}\ba
|J_1|&\le C\|\sqrt\n \te_t\|_{L^2}\|\n\bar x\|_{L^\infty}^{1/2}\|u \bar x^{-{1/2}}\|_{L^\infty}\|\na\te\|_{L^2}\\
&\le \frac{c_\upsilon}{2}\|\sqrt\n \te_t\|_{L^2}^2+C\psi^A,
\ea\ee
\be\la{r19}\ba
|J_2|&\le C \|\sqrt\n\te_t\|_{L^2}\|\n^{1/2}\te\|_{L^6}\|\na u\|_{L^2}^{2/3}\|\na u\|_{H^1}^{1/3}\\
&\le \frac{c_\upsilon}{2}\|\sqrt\n\te_t\|_{L^2}^2+C\psi^A,
\ea\ee
and
\be\la{r20}\ba
|J_3|
&\le \|\na u_t\|_{L^2}\|\na u\bar x^{\al/2}\|_{L^2}^{1/2}\|\te\bar x^{-{\al/4}}\|_{L^{16/\al}}\|\na u\|_{L^2}^{(4-\al)/8}\|\na u\|_{H^1}^{\al/8}\\
&\le C\psi^A\|\na u_t\|_{L^2}.
\ea\ee
Submitting \eqref{r18}--\eqref{r20} into \eqref{b33}, one derives
\be\la{r21}\ba
\frac{d}{dt}\int(\ka|\na\te|^2+\de|\te|^2-2 F(\na u)\te)dx+c_\upsilon\int\rho|\te_t|^2dx\le C\psi^A(\|\na u_t\|_{L^2}+1).
\ea\ee

{\it Step 3}
Multiplying \eqref{a1}$_3$ by $2\rho^2 \theta_t$ and integrating the resulting equality  yield  that
\be\la{b38}\ba
&\frac{d}{dt} \int (\ka\rho^2 |\na\te|^2+\de\n^2|\te|^2)dx + 2c_\upsilon\int \rho^3 | {\te}_t |^2dx\\
&=2\int \n\n_t(\ka|\na\te|^2+\de|\te|^2)dx-4\ka\int\na\n\cdot \na \te\rho\theta_t dx-2c_\upsilon \int \rho^3 u\cdot\na\theta\theta_t dx \\
&\quad -2R\int \rho^3 \te \div u \theta_t dx +2\int F(\na u) \rho^2 \theta_t dx\triangleq \sum_{i=1}^5 W_i.
\ea\ee
Direct calculations show  that
\be\la{b39}\ba
|W_1|&\le C\|\n\|_{L^\infty}\|\n_t\|_{L^2}(\|\na\te\|_{L^2}\|\na\te\|_{H^1}+\de\|\te\|_{L^2}\|\te\|_{H^1})\\
   &\le C\psi^A(\|\n\|_{L^\infty}\|\na u\|_{L^2}+\|\na\n\bar x\|_{L^q}\|u\bar x^{-1}\|_{L^{2q/(q-2)}})(\|\na\te\|_{H^1}+1)\\
   &\le C\psi^A(\|\sqrt\n\te_t\|_{L^2}+1),
\ea\ee
\be\ba\la{b40}
|W_2|
	&\le C\|\n\|_{L^\infty}^{1/2}\|\sqrt\n\te_t\|_{L^2}\|\na\n\|_{L^q}\|\na\te\|_{L^2}^{(q-2)/q}\|\na\te\|_{H^1}^{2/q}\\
	&\le C\psi^A(\|\sqrt\n\te_t\|_{L^2}^{(q+2)/q}+1),
\ea\ee
and
\be\ba\la{b41}
\sum_{i=3}^5|W_i|&\le c_\upsilon\int\n^3|\te_t|^2dx+C\int (\n^3 (|u||\na\te|+\te|\na u|)^2+ \n|\na u|^4)dx\\
   &\le c_\upsilon\int\n^3|\te_t|^2dx+C\|\n\|_{L^\infty}\|\n u\|_{L^4}^2\|\na\te\|_{L^2}\|\na \te\|_{H^1} \\
   &\quad +C\|\n\|_{L^\infty}\|\n\te\|_{L^4}^2\|\na u\|_{L^2}\|\na u\|_{H^1} + C\|\n\|_{L^\infty}\|\na u\|_{L^2}^2\|\na u\|_{H^1}^2\\
   &\le c_\upsilon\int\n^3|\te_t|^2dx +C\psi^A(\|\sqrt\n\te_t\|_{L^2}+1), 
\ea\ee
owing to \eqref{uu}, \eqref{uu1}, 
\eqref{b04}, and \eqref{b05}. 
Putting \eqref{b39}--\eqref{b41} into \eqref{b38}  leads to
\be \ba \la{b42}
&\frac{d}{dt} \int (\ka\rho^2 |\na\te|^2+\de\n^2|\te|^2)dx+c_\upsilon \int\n^3|\te_t|^2dx\le  C\psi^A(\|\sqrt\n\te_t\|_{L^2}^{(q+2)/q}+1). 
\ea\ee

{\it Step 4}
Adding \eqref{r21} and \eqref{b42} both multiplied by $c_\upsilon^{-1}(\tilde C+1)$ to \eqref{b32} and applying Young's inequality to the resulting inequality conclude that
\be\ba\la{b44}
&c_\upsilon^{-1}(\tilde C+1)\frac{d}{dt}\int(\ka|\na\te|^2+\de|\te|^2-2 F(\na u)\te+\ka\rho^2 |\na\te|^2+\de\n^2|\te|^2)\\
&+\frac{d}{dt}\int\rho|u_t|^2dx+\frac{1}{2}\int(\mu|\na u_t|^2+\n|\te_t|^2+\n^3|\te_t|^2)dx\le C\psi^A.
\ea\ee
Taking into account the compatibility condition \eqref{co1}, we can define
\be\la{qq1}\sqrt\n u_t(x,t=0)=-(\n_0)^{1/2}u_0\cdot \na u_0-g.\ee
Therefore, using \eqref{b44}, \eqref{qq1},  \eqref{term},  \eqref{b155}, and choosing $p'=30q/(q-2)$, one gets 
\be\ba\notag
&\psi_3(T)+\int_0^T(\|\na u_t\|_{L^2}^2+\|\sqrt\n\te_t\|_{L^2}^2)dt\\
&\le C+C\int_0^T\psi^Adt+C\psi_1^2(T)\psi_2^2(T)\psi_3^{7/8}(T)\\
&\le \left(C+C(p')\psi_3^{2q/(2q+p'q-2p')+7/8}(T)\right)\left(1+\int_0^T\psi^Adt\right)^4\\
&\le \frac{1}{2}\psi_3(T)+C+C\int_0^T\psi^Adt.
\ea\ee
 The proof of Lemma \ref{le4} is completed.
\end{proof}

\begin{lemma}\la{lee7}
Let  $q$ be as in Theorem \ref{th1} and $T_1$ as in Lemma \ref{lee0}. Then there exist positive  constants $C$ and $A\geq1$ such that for any $0<T\le T_1$,
\be\ba\la{bxj}
\sup_{0\le t\le T}\|\na(\bar x^a\n)\|_{L^{q}}\le C+C\int_0^T\psi^Adt.
\ea\ee
\end{lemma}
\begin{proof}
It follows from the standard calculation, \eqref{uu}, \eqref{b111}, and \eqref{r2} that
 \be\ba\notag
\frac{d}{dt}\|\na \omega\|_{L^{q}}
&\le C(\|\na u\|_{L^\infty}+\|u\cdot\na \log \bar x\|_{L^\infty})\|\na \omega\|_{L^{q}}\\
&\quad+C\|\omega\|_{L^\infty}(\|\na^2u \|_{L^{q}}+\||\na u||\na\log\bar x|\|_{L^{q}}+\||u||\na^2\log \bar x|\|_{L^{q}})\\
&\le C\psi(\|\na u\|_{W^{1,q}}+\|u\cdot\na \log \bar x\|_{L^\infty}+\||u||\na^2\log \bar x|\|_{L^{q}})\\
&\le C\psi^A(\|\na u_t\|_{L^2}+\|\sqrt\n \te_t\|_{L^2}+1).
\ea\ee
Then integrating the above inequality with respect to $t$ and using \eqref{b47} show \eqref{bxj} and finish the proof of Lemma \ref{lee7}
\end{proof}

Now,  we are ready to prove Proposition \ref{pr2}.

{\bf Proof of Proposition \ref{pr2}.}
For $T_1$ as in Lemma \ref{lee0}, by virtue of Lemma  \ref{le01} with $p'=3$ and Lemmas \ref{le3}--\ref{lee7}, one  gets that for any $0<T\le T_1$,
\be\ba\notag
\psi(T)\le C_1\left(1+\int_0^T\psi^A dt\right),\ea\ee
with $\psi(t)$ defined in \eqref{bb0}.
In fact, since the the constants $A$ and $C_1$ are independent of $M_1$, we can choose $ M_1=4C_1$. Then for $T^\ast\triangleq \min\{T_1,{M}_1^{-A}\}$, it holds
\be\ba\notag \psi(T^\ast)\le  M_1/2,\ea\ee
which along with  \eqref{hii}, \eqref{b47}, \eqref{b04}, \eqref{b05}, and \eqref{r2} concludes \eqref{d03} and completes the proof of Proposition \ref{pr2}.
\thatsall

\begin{lemma}\la{le6}
Let $T^\ast$ be as in Proposition \ref{pr2}. Then the following estimate holds
\be\ba\la{b52}
\sup_{0\le t\le T^\ast}t\left(\|\sqrt\n\te_t\|_{L^2}^2+\|\na^2\te\|_{L^2}^2\right)\!+\!\int_0^{T^\ast}t(\|\na\te_t\|_{L^2}^2+\de\|\te_t\|_{L^2}^2+\|\na^2\te\|_{{H^1}}^2)dt\le C.
\ea\ee
\end{lemma}
\begin{proof}
First, differentiating    \eqref{a1}$_3$ with respect to $t$ yields
\be\ba\la{b53}
&c_\upsilon\n(\te_{tt}+u\cdot\na \te_t)-\ka\Delta\te_t+\de\te_t+c_\upsilon\n_t(\te_{t}+u\cdot\na \te)+c_\upsilon\n u_t\cdot\na\te\\
&+P_t\div u+P\div u_t-F(\na u)_t=0.
\ea\ee
Multiplying (\ref{b53}) by $ \te_t$ and integrating the resulting equality over $\r^2$, we obtain that
\be\la{b54}\ba
& \frac{c_\upsilon}{2}\frac{d}{dt}\int \n |\te_t|^2dx +\int( \ka|\na\te_t|^2 +\de|\te_t|^2)dx\\
&= -c_\upsilon\int\n_t|\te_t|^2dx-   \int\n_t\te_t(c_\upsilon u\cdot\na\te+R\te\div u) dx-R\int \n\te\te_t\div u_t dx\\
&\quad-   \int\n\te_t(c_\upsilon u_t\cdot\na\te+R\te_t\div u) dx+\int F(\na u)_t\te_t dx\triangleq \sum_{i=1}^5L_i.
\ea\ee
It follows from \eqref{bb1}--\eqref{uu1},  \eqref{b111}, \eqref{b05}, \eqref{b112},  and Proposition \ref{pr2}  that for any  $t\in(0,T^\ast]$, 
\be\ba\la{b55}
|L_1|&\le C\int \n|u||\te_t||\na\te_t|dx\\
     &\le C\|\na\te_t\|_{L^2}\|\sqrt\n \te_t\|_{L^2}\|\n\bar x\|_{L^\infty}^{1/2}\|u\bar x^{-{1/2}}\|_{L^\infty}\\
     &\le \epsilon\|\na\te_t\|_{L^2}^2+C(\epsilon)\|\sqrt\n \te_t\|_{L^2}^2,
\ea\ee
\be\ba\la{b56}
|L_2|&\le C\int\n|u||\te_t|(|\na u||\na\te|+|u||\na^2\te|+\te|\na^2 u|)dx\\
     &\quad+C\int\n|u||\na\te_t|(|u||\na\te|+\te|\na u|)dx\\
     &\le C\|\sqrt\n \te_t\|_{L^2}\|\rho^{1/4} u\|_{L^{\infty}}(\|\rho^{1/4} u\|_{L^{\infty}}\|\na^2\te\|_{L^2}+\|\rho^{1/4} \te\|_{L^{\infty}}\|\na^2 u\|_{L^2})\\
     &\quad+C\|\sqrt\n \te_t\|_{L^2}\|\n\bar x\|_{L^\infty}^{1/2}\|u\bar x^{-{1/2}}\|_{L^\infty}\|\na u\|_{L^4}\|\na\te\|_{L^2}^{1/2}\|\na\te\|_{H^1}^{1/2}\\
     &\quad+C\|\na\te_t\|_{L^2}\|\rho^{1/2} u\|_{L^6}(\|\rho^{1/2} u\|_{L^6}\|\na\te\|_{L^6}+\|\rho^{1/2} \te\|_{L^6}\|\na u\|_{L^6})\\
     &\le \epsilon\|\na\te_t\|_{L^2}^2+C(\epsilon)(\|\sqrt\n \te_t\|_{L^2}^2+1),
\ea\ee
\be\ba\la{b57}
|L_3|&\le C\|\na u_t\|_{L^2}\|\n\bar x\|_{L^6}\| \te\bar x^{-{1/2}}\|_{L^6}\| \te_t\bar x^{-{1/2}}\|_{L^6}\\
     &\le C\|\na u_t\|_{L^2}(\|\sqrt\n \te_t\|_{L^2}+\|\na\te_t\|_{L^2})\\
     &\le \epsilon\|\na\te_t\|_{L^2}^2+C(\epsilon)(\|\na u_t\|_{L^2}^2+\|\sqrt\n \te_t\|_{L^2}^2),
\ea\ee
\be\ba\la{b58}
|L_4|&\le C\|\te_t\bar x^{-{1/2}}\|_{L^6}\|\n\bar x\|_{L^\infty}^{1/2}(\|\sqrt\n\te_t\|_{L^2}\|\na u\|_{L^3}+\|\sqrt\n u_t\|_{L^2}\|\na \te\|_{L^3})\\
    &\le C(\|\sqrt\n \te_t\|_{L^2}+\|\na\te_t\|_{L^2})(\|\sqrt\n \te_t\|_{L^2}+\|\na\te\|_{H^1}^{1/3})\\
    &\le \epsilon\|\na\te_t\|_{L^2}^2+C(\epsilon)(\|\sqrt\n \te_t\|_{L^2}^2+1),
\ea\ee
and
\be\ba\la{b59}
|L_5|
&\le C\|\na u_t\|_{L^2}\|\na u\bar x^{\al/2}\|_{L^2}^{1/2}\|\te_t\bar x^{-{\al/4}}\|_{L^{16/\al}}\|\na u\|_{L^2}^{(4-\al)/8}\|\na u\|_{H^1}^{\al/8}\\
&\le C(\|\sqrt\n\te_t\|_{L^2}+\|\na\te_t\|_{L^2})\|\na u_t\|_{L^2}\\
&\le \epsilon\|\na\te_t\|_{L^2}^2+C(\epsilon)(\|\sqrt\n\te_t\|_{L^2}^2+\|\na u_t\|_{L^2}^2).
\ea\ee
Substituting \eqref{b55}--\eqref{b59} into \eqref{b54} and choosing $\epsilon$ suitably small, one gets
\be\la{lg}\ba
c_\upsilon\frac{d}{dt}\|\sqrt\n\te_t\|_{L^2}^2 +\ka\|\na\te_t\|_{L^2}^2+\de\|\te_t\|_{L^2}^2\le C(\|\sqrt\n \te_t\|_{L^2}^2+\|\na u_t\|_{L^2}^2+1).
\ea\ee

Multiplying \eqref{lg} by $t$, integrating the resulting inequality with respect to $t$ over $(0,T^\ast)$, and using Proposition \ref{pr2} infer
\be\ba\la{b522}
\sup_{0\le t\le T^\ast}t\|\sqrt\n\te_t\|_{L^2}^2
+\int_0^{T^\ast}t(\|\na\te_t\|_{L^2}^2+\de\|\te_t\|_{L^2}^2)dt\le C.
\ea\ee

Moreover,  we deduce from the standard $L^{2}$-estimate to \eqref{a1}$_3$, Proposition \ref{pr2}, \eqref{bb1}--\eqref{uu1}, \eqref{b111}, and \eqref{b05}--\eqref{b112} that
\be\ba\la{der}
&\|\na^2\te\|_{H^{1}}\\
&\le C(\|\na (\n(\te_t+ u \cdot \na \te)) \|_{L^{2}} +\||\na u||\na^2 u|\|_{L^{2}}+\|\na(P\div u)\|_{L^{2}})\\
&\quad+C(\de\|\na\te\|_{L^{2}}+\|\na\te\|_{H^{1}})\\
& \le C(\|\na\n\bar x\|_{L^q}\|\te_t\bar x^{-1}\|_{L^{2q/(q-2)}}+\|\n\|_{L^\infty}\|\na\te_t\|_{L^2})+C\|\n\|_{L^{\infty}}\|\na u\|_{L^4}\|\na\te\|_{L^4}\\
&\quad+C\|\na\n\bar x\|_{L^q}(\|u\bar x^{-1}\|_{L^\infty}\|\na \te\|_{L^{2q/(q-2)}}+\|\te\bar x^{-1}\|_{L^\infty}\|\na u\|_{L^{2q/(q-2)}})\\
&\quad+C\|\n\bar x\|_{L^\infty}(\|u\bar x^{-1}\|_{L^\infty}\|\na ^2\te\|_{L^{2}}+\|\te\bar x^{-1}\|_{L^\infty}\|\na^2 u\|_{L^{2}})\\
&\quad+C\|\na u\|_{L^\infty}\|\na^2 u\|_{L^2}+C\|\sqrt\n\te_t\|_{L^2}+C\\
&\le C(\|\te_t\bar x^{-1}\|_{L^{2q/(q-2)}}+\|\na\te_t\|_{L^2}+\|\na\te\|_{H^1}+\|\sqrt\n\te_t\|_{L^2}+\|\na u\|_{W^{1,q}}+1)\\
&\le C(\|\sqrt\n\te_t\|_{L^2}+\|\na u_t\|_{L^2}+\|\na\te_t\|_{L^2}+1),\ea\ee
which as well as  \eqref{b05}, \eqref{d03}, and \eqref{b522} shows \eqref{b52}.
The proof of Lemma \ref{le6} is finished.
\end{proof}


\subsection{The vanishing damping of $\te$}\la{sec3.3}
With all the a priori estimates in Subsection \ref{sec3.2}  in mind, we are ready to prove Theorem \ref{th1}.

{\bf Proof of Theorem \ref{th1}.} Let $(\n_0\geq 0,u_0,\te_0\geq 0)$ satisfying conditions \eqref{q4} and  \eqref{co1} be the initial data as described in Theorem \ref{th1}. Without loss of generality, assume that there exists a constant $N_1\geq 1$ such that 
$$\int_{B_{N_1}}\n_0 dx\geq 1/2.$$
We construct approximate initial data $(\n_0^\de,u_0^\de,\te_0^\de)=(\n_0,u_0^\de,\te_0^\de)$. 

In fact, since $\na \te_0\in L^2$, there exists a sequence $v^\de\in C_0^\infty$ such that
\be\la{sb}\lim_{\de\rightarrow 0}\|v^\de-\na\te_0\|_{L^2}=0.\ee
Then we choose $\te_0^\de$ as the unique smooth solution to the following elliptic equation:
\be\notag-\Delta\te_0^\de+\de\te_0^\de+\n_0^\de\te_0^\de=-\div v^\de+\n_0\te_0,\quad\text{in}\,\r^2.\ee

According to  the procedures in \cite{liliang}, \eqref{sb}, and \eqref{q4}, there exists a positive constant $C$ independent of $\de$ such that
\be\la{47} \|\na \te^\de_0\|_{L^2}^2+\de\|\te_0^\de\|_{L^2}^2+\|\sqrt{\n_0^\de} \te_0^\de\|_{L^2}^2\le C,\ee
and
\be\ba\la{57} &\lim_{\de\rightarrow0}\left(\|\na \te_0^\de-\na \te_0\|_{L^2}+\|\sqrt{\n_0^\de}\te_0^\de-\sqrt{\n_0}\te_0\|_{L^2}\right)=0.
\ea\ee

Next, consider $u_0^\de$ as the unique smooth solution to the following elliptic equation:
\be\notag-\mu\Delta u_0^\de-(\mu+\lambda)\na\div u_0^\de+R\na(\n_0^\de\te_0^\de)+\n_0^\de u_0^\de=\n_0^{1/2}g+\n_0 u_0,\quad \text{in}\,\r^2.\ee
Analogously, by \eqref{q4}, \eqref{co1}, \eqref{47}, and \eqref{57}, we derive
\be\la{u47} \|\na u^\de_0\|_{H^1}^2+\|\sqrt{\n_0^\de}u_0^\de\|_{L^2}^2\le C,\ee
and
\be\ba\la{u57}\lim_{\de\rightarrow0}\left(\|\na u_0^\de-\na u_0\|_{H^1}+\|\sqrt{\n_0^\de}u_0^\de-\sqrt{\n_0}u_0\|_{L^2}\right)=0.
\ea\ee

Moreover, for initial data $(\n_0^\de,u_0^\de,\te_0^\de)$,  $g$ defined in  \eqref{co1} in fact is replaced by $g^\de\triangleq g+\sqrt{\n_0}u_0-\sqrt{\n_0^\de}u_0^\de$ satisfying
\be\ba\la{pp}\|g^\de\|_{L^2}\le C,\ea\ee
due to \eqref{q4}, \eqref{co1}, and \eqref{u47}.

Now, for the approximate initial data $(\n_0^\de,u_0^\de,\te_0^\de)$ satisfying  \eqref{47}, \eqref{u47}, and \eqref{pp} with fixed $\de>0$, we deduce from Lemma \ref{th3} that there exists a strong solution $(\n^\de,u^\de,\te^\de)$ of (\ref{a1}) (\ref{q2}) (\ref{q3}) on $\r^2\times (0,T_\de]$ for some $T_\de>0$. Furthermore, it follows from Proposition \ref{pr2} and Lemma \ref{le6}  that there exist positive constants $T^\ast$, $M_1$, $\tilde M_1$, and $C$ independent of $\de$ such that \eqref{d01}, \eqref{d03}, and \eqref{b52} all hold with $(\n,u,\te)$ being replaced by $(\n^\de,u^\de,\te^\de)$. Then passing to the limit $\de\rightarrow  0$ together with \eqref{57}, \eqref{u57}, and standard compactness arguments leads to that there exists a strong solution  $(\n,u,\te)$ of the Cauchy problem (\ref{q1}) (\ref{q2}) (\ref{q3}) on $\r^2\times (0,T^\ast]$ satisfying \eqref{q5} and \eqref{q7} except that   $\te\in L^2(0,T^\ast;L^{ q_2})$.

Indeed, 
 proceeding as the procedures in Lemma \ref{lee0} , we can  obtain the norm of $\te$ itself smoothly.  To be precise, adding \eqref{q1}$_2$ multiplied by $u$ to \eqref{q1}$_3$ and using \eqref{q1}$_1$ show that
\be\la{m0}\ka\Delta\te=\n\dot w-\frac{\mu}{2}\Delta |u|^2-\mu\div(u\cdot\na u)-\lambda\div(u\div u)+\div(Pu),\ee
with $w=c_\nu \te+|u|^2/2$,
which together with \eqref{q5} implies that for $t\in (0,T^\ast)$,
\be\la{m1}\int \n\dot w (x,t)dx=0.\ee
Then \eqref{m0} leads to
\be\ba\notag
\ka\na\te&=-\na (-\Delta)^{-1}(\n\dot w)-\frac{\mu}{2}\na |u|^2+\na (-\Delta)^{-1}\div(\mu u\cdot\na u+\lambda u\div u-Pu),\ea\ee
which is equivalent to
$$\ka\na \te\triangleq \tilde K\ast (\n\dot w)-\frac{\mu}{2}\na |u|^2+\tilde{\mathcal{T}}(\mu u\cdot\na u+\lambda u\div u-Pu),$$
 where $\tilde{\mathcal{T}}=\na(-\Delta)^{-1}\div$ is the singular integral operator of convolution type with Calderon-Zygmund kernel which is bounded in $L^p\,(1<p<\infty)$ and the kernel $\tilde K(x)$ satisfies
\be\notag|\tilde K(x)|\le C|x|^{-1},\quad|\na\tilde K(x)|\le C|x|^{-2}.\ee
Now for any $x\neq0$, $t\in (0,T^\ast)$, and $\al\triangleq \min\{1,a/2\}$, by \eqref{m1} and the method in Lemma \ref{lee0}, it holds
\be\notag\ba
|\na \te(x,t)
&|\le C|x|^{-\alpha}\mathcal{I}_1(|y|^\alpha|\n\dot w|)+C|x|^{-1-\alpha}\||y|^\alpha \n\dot w\|_{L^1}\\
&\quad+C|u||\na u|+|\tilde {\mathcal{T}}(\mu u\cdot\na u+\lambda u\div u-Pu)|.
\ea\ee
Hence, for $q_3\triangleq 3/(1+\al)\in[3/2,2)$, one has
\be\ba\la{m2}
\|\na\te\|_{L^{q_3}}
&\le \|\na\te\|_{L^{q_3}(B_1)}+C\||x|^{-\alpha}\mathcal{I}_1(|y|^\alpha|\n\dot w|)\|_{L^{q_3}(\r^2\setminus{B_1})}\\
&\quad +C\||x|^{-1-\alpha}\|_{L^{q_3}(\r^2\setminus{B_1})}\||y|^\alpha \n\dot w\|_{L^1}+C(\||u||\na u|\|_{L^{ q_3}}+\|\n\te u\|_{L^{ q_3}})\\
&\le C\|\na\te\|_{L^2}+C\||x|^{-\alpha}\|_{L^{{3/\al}}(\r^2\setminus B_1)}\|\mathcal{I}_1(|y|^\alpha|\n\dot w|)\|_{L^3}\\
&\quad +C\||y|^\alpha \n\dot w\|_{L^1}+C\|u\|_{L^{q_1}}(\|\na u\|_{L^{12/(4+\al)}}+\|\n\te\|_{L^{12/(4+\al)}})\\
&\le C+C\||y|^\alpha |\n\dot w|\|_{L^{6/5}}+C\||y|^\alpha \n\dot w\|_{L^1}\\
&\le C(\|\sqrt\n\te_t\|_{L^2}+1),
\ea\ee
where we have used \eqref{q5} and the following fact: for any $p\in[1,2)$,
\be\ba\notag
&\||y|^\al \n\dot w\|_{L^p}\\
&\le C(\||y|^\al\n \te_t\|_{L^p}+\||y|^\al\n u\cdot\na \te\|_{L^p}+\||y|^\al\n |u|^2_t\|_{L^p}+\||y|^\al\n u\cdot\na |u|^2\|_{L^p})\\
&\le C\|\sqrt\n \te_t\|_{L^2}\|\n\bar y^{2\al}\|_{L^{p/(2-p)}}^{1/2}+C\|\na \te\|_{L^2}\|\n\bar y^\al\|_{L^{2p/(2-p)}}\|u\|_{L^\infty}\\
&\quad+C\|\sqrt\n u_t\|_{L^2}\|\n\bar y^{2\al}\|_{L^{p/(2-p)}}^{1/2}\|u\|_{L^\infty}+C\|\na u\|_{L^2}\|\n\bar y^\al\|_{L^{2p/(2-p)}}\|u\|_{L^\infty}^2\\
&\le C(\|\sqrt\n \te_t\|_{L^2}+1),
\ea\ee
due to \eqref{q5}.
Combining \eqref{m2} with \eqref{q5} yields that for $q_2\triangleq 2q_3/(2-q_3)=6/(2\al-1)$,
\be\notag\int_0^{T^\ast}\|\te\|_{L^{q_2}}^2dt\le C\int_0^{T^\ast}\|\na\te\|_{L^{q_3}}^2dt\le C. \ee

Finally, since the proof of the uniqueness of $(\n,u,\te)$ is similar to that in \cite{liliang}, it will be omitted for simplicity.  The proof of Theorem \ref{th1} is finished.
\thatsall

\section{Proof of Theorem \ref{th2}}\la{sec4}
Let $(\n,u,\te)$ be a strong solution to the Cauchy problem \eqref{q1} \eqref{q2} \eqref{q3} on $\r^2\times(0,T^\ast]$ described in Theorem \ref{th1}. Suppose that \eqref{qq6} were false, that is,
\be\la{t3}\lim_{T\rightarrow T_\ast}\|\div u\|_{L^1(0,T;L^\infty)}\le M_0<+\infty.\ee
This as well as \eqref{q1}$_1$ immediately leads to the following  upper bound of the density  (see \cite[Lemma 3.4]{hlx}).
\begin{lemma}\la{le00}
Under the condition (\ref{t3}), it holds that for $0\le T< T_\ast$,
\be\la{t4} \sup_{0\le t\le T}\|\n\|_{L^\infty}\le C,\ee
where (and in what follows) $C,\,C_i\,(i=2,\cdots,5)$ denote some generic constants depending only on $\mu$, $\lambda$, $R,$ $c_\upsilon,\,\ka,\,q,\,a,\,\eta_0,\,T_\ast,\,M_0$, and the initial data.
\end{lemma}

\begin{lemma}\la{le55} Under the condition (\ref{t3}), it holds that for $0\le T< T_\ast$,
\be \la{a2.12}\sup_{0\le t\le T}\left( \|\n\te\|_{L^1}+\|\sqrt\n u\|_{L^2}^2\right)+\int_0^T\|\na u\|_{L^2}^2 dt\le C.\ee
\end{lemma}
\begin{proof}
Adding $(\ref{q1})_2$ multiplied by $2u$ to $(\ref{q1})_3$, we obtain
after integrating the resulting equality over $\r^2$    that
\be\notag\ba
\frac{d}{dt} \int\left(c_\upsilon\n\te+ \n |u|^2\right)dx+\mu\|\na u\|_{L^2}^2+(\mu+\lambda)\|\div u\|_{L^2}^2\le R\|\div u\|_{L^\infty}\int\n\te dx,\ea\ee
which along with Gronwall's inequality and \eqref{t3} yields \eqref{a2.12} directly.
 The proof of Lemma \ref{le55} is  finished.
\end{proof}

\begin{lemma}\la{le7}
Under the condition (\ref{t3}), it holds that for $0\le T< T_\ast$,
\be \la{a2.1}\sup_{0\le t\le T}\left( \|\sqrt\n\te\|_{L^2}^2+\|\na u\|_{L^2}^2\right)+\int_0^T\left(\|\na \te\|_{L^2}^2+\|\sqrt\n\dot u\|_{L^2}^2\right) dt\le C.\ee
\end{lemma}

\begin{proof}
First, multiplying \eqref{q1}$_3$ by $2\te$ and integrating the resulting equality by parts lead to
\be\la{c1}\ba  &c_\upsilon\frac{d}{dt}\int \n\te^2dx+2\ka\|\na\te\|_{L^2}^2\le C\|\div u\|_{L^\infty}\int\n\te^2 dx+C\int |\na u|^2\te dx.\ea\ee

Next, multiplying $(\ref{q1})_2$  by $2u_t $ and integrating the resulting equality over $\r^2,$  one deduces from (\ref{hj1}) and \eqref{t4} that
\be\ba \la{hh17}
&\frac{d}{dt}\int \left(  {\mu}|\na u|^2+ (\mu+\lambda)(\div u)^2\right)dx+ \frac{3}{2}\int\rho|\dot u|^2dx \\
&\le  2\frac{d}{dt}\int  P \div u  dx-2\int P_t \div udx+C\int \n|u\cdot \na u|^2dx\\
&\le 2\frac{d}{dt}\int  P \div u  dx-\frac{1}{2\mu+\lambda}\frac{d}{dt}\int P^2 dx-\frac{2}{2\mu+\lambda}\int P_t Gdx+ C\|u\|_{L^\infty}^2\|\na u\|_{L^2}^2.\ea\ee

Noticing that (\ref{q1})$_3$ is equivalent to
\be \notag c_\upsilon\ba P_t=-c_\upsilon\div (Pu)-RP\div u+R\ka\Delta\te+RF(\na u),
 \ea\ee
which combined with \eqref{t4} and \eqref{h19} infers
\be\la{a16}\ba
\frac{2}{2\mu+\lambda}\left|\int   P_t Gdx\right| 
& \le C\int (P|u|+|\na\te|)|\na G|dx+ \int( P|\div u|+F(\na u))|G|dx
\\ &  \le  C (\|\n\te u\|_{L^2}+\|\na \te\|_{L^2})\|\na G\|_{L^2} \\
&\quad+ C \|\div u\|_{L^\infty}(\|\sqrt\n\te\|_{L^2}^2+ \|\na u\|_{L^2}^2)+C\int  |\na u|^2\te dx\\
&  \le \frac{1}{2} \|\sqrt\n  \dot u\|^2_{L^2} + C_2\|\na \te\|_{L^2}^2+C\int  |\na u|^2\te dx\\
&\quad+C( \|\div u\|_{L^\infty}+\|u\|_{L^\infty}^2)(\|\sqrt\n\te\|_{L^2}^2+ \|\na u\|_{L^2}^2).
 \ea\ee
Submitting \eqref{a16}  into \eqref{hh17}, one has
\be\ba \la{hh7}
&B'(t)+ \|\sqrt\n \dot u\|_{L^2}^2\\
&\le  C_2\|\na \te\|_{L^2}^2+C( \|\div u\|_{L^\infty}+\|u\|_{L^\infty}^2)(\|\sqrt\n\te\|_{L^2}^2+ \|\na u\|_{L^2}^2)+C\int  |\na u|^2\te dx,
\ea\ee
where $$B(t)\triangleq \int\left({\mu}|\na u|^2+ (\mu+\lambda)(\div u)^2 +\frac{1}{2\mu+\lambda}P^2-2P \div u\right) dx$$
satisfies
\be\la{h1} \mu\|\na u\|_{L^2}^2/2-C_3\|\sqrt\n\te\|_{L^2}^2\le B(t)\le C(\|\sqrt\n\te\|_{L^2}^2+ \|\na u\|_{L^2}^2),\ee
due to \eqref{t4}. 

To estimate the last term in \eqref{c1} and \eqref{hh7}, multiplying \eqref{q1}$_2$ by $u\te$ and using integration by parts and \eqref{t4}, it holds  that for any $\epsilon\in(0,1),$
\be\la{c4}\ba
&\int ({\mu}|\na u|^2+ (\mu+\lambda)(\div u)^2)\te dx\\
&\le C\int\n\te|u|(|\dot u|+|\na\te|)dx+C\int|u||\na u||\na\te|dx+C\int\n\te^2|\div u|dx\\
&\le \epsilon(\|\na\te\|_{L^2}^2+\|\sqrt\n \dot u\|_{L^2}^2)+C(\epsilon)( \|\div u\|_{L^\infty}+\|u\|_{L^\infty}^2)(\|\sqrt\n\te\|_{L^2}^2+ \|\na u\|_{L^2}^2).
 \ea\ee

Therefore, adding  \eqref{c1} multiplied  by $C_4\triangleq \max\{c_\upsilon^{-1}(C_3+1),(2\ka)^{-1}(C_2+1)\}$ to \eqref{hh7} and using \eqref{c4}, \eqref{h1}, 
  one derives after choosing $\epsilon$ small enough that
\be\la{c3}B'_1(t)+\frac{1}{2}B_2\le C( \|\div u\|_{L^\infty}+\|u\|_{L^\infty}^2)B_1+C,\ee
with
\be\la{xy}
B_1(t)\triangleq e+C_4 c_\upsilon\|\sqrt\n\te\|_{L^2}^2+B(t)\geq e+\mu\|\na u\|_{L^2}^2/2+\|\sqrt\n\te\|_{L^2}^2,\ee
\be\notag B_2(t)\triangleq e+\|\na\te\|_{L^2}^2+\|\sqrt\n \dot u\|_{L^2}^2.\ee


It remains to estimate $\|u\|_{L^\infty}^2$.
In fact, $\eqref{q1}_1$ implies that for any $t\in(0,T_\ast)$,
$\|\n\|_{L^1}(t)=\|\n_0\|_{L^1},$
which along with \eqref{a2.12} and  the method in \cite{liliang} leads to
\be\la{t11}\inf_{0\le t\le T}\int_{B_{N_0}}\n dx\geq \frac{1}{2}\|\n_0\|_{L^1},\ee for some suitably large $N_0\geq 1$.
  Thus, by virtue of Lemma \ref{w5},  \eqref{t11},  and \eqref{t4}, for  any $\ve> 0$ and $\eta\in(0,1],$  every $v\in   \ti D^{1,2}(\r^2) $  satisfies
  \be \la{cc1}\ba  \|v\bar x^{-1}\|_{L^2}+\|v\bar x^{-\eta}\|_{L^{(2+\ve)/\eta} } &\le C(\epsilon,\eta)(\|\sqrt{\n}v\|_{L^2}+\|\na v\|_{L^2}).\ea\ee
We thus deduce from \eqref{b9}, \eqref{t4}, \eqref{a2.12}, and \eqref{cc1} that
\be\ba\notag
\frac{d}{dt}\|\n\bar x^{a}\|_{L^1}
&\le C\|\n\bar x^{a-{2/3}+{1/(1+a)}}\|_{L^{3(1+a)/(2+3a)}}\|u\bar x^{-{1/(1+a)}}\|_{L^{3(1+a)}}\\
&\le C\|\n\bar x^{a}\|_{L^1}^{(2+3a)/(3(1+a))}(1+\|\na u\|_{L^2}),\ea\ee
which as well as Gronwall's inequality and \eqref{a2.12} gives
\be\la{ch}\sup_{0\le t\le T}\|\n\bar x^{a}\|_{L^1}\le C.\ee

Note that by adopting the  method in Lemma \ref{lee0}, \eqref{key3} and \eqref{k1} both still hold in this section. In fact, for $\al\triangleq\min\{a/2,1\}$, taking  $\tilde \al=\al/2$ in \eqref{k1} and using \eqref{t4}, \eqref{cc1}, and \eqref{ch}, we obtain  that for $\beta=\al/3$,  
\be\ba\la{yy}
&\|\na u|x|^{\beta}\|_{L^2}\\
&\le \|\na u\|_{L^2(B_1)}+C\||x|^{\beta-\al/2}\mathcal{I}_1(|y|^{\al/2}|\n\dot u|)\|_{L^2(\r^2\setminus B_1)}\\
&\quad+C\||x|^{\beta-1-\al/2}\|_{L^2(\r^2\setminus B_1)}\||y|^{\alpha/2}\n\dot u\|_{L^1}+\|\mathcal{T}_1 P|x|^{\beta}\|_{L^2}+\|\mathcal{T}_2 P|x|^{\beta}\|_{L^2}\\
&\le \|\na u\|_{L^2}+C\||x|^{-\al/6}\|_{L^{14/\al}(\r^2\setminus B_1)}\|\mathcal{I}_1(|y|^{\al/2}|\n\dot u|)\|_{L^{14/(7-\al)}}\\
&\quad+C\||y|^{\al/2}\n\dot u\|_{L^1}+C\| P|x|^{\beta}\|_{L^2}\\
&\le \|\na u\|_{L^2}+C(\||y|^{\al/2}\n\dot u\|_{L^{14/(14-\al)}}+\||y|^{\al/2}\n\dot u\|_{L^1})+C\|P|x|^{\al/3}\|_{L^2}\\
&\le \|\na u\|_{L^2}+C\|\sqrt\n\dot u\|_{L^2}(\|\bar y^{7\al/(7-\al)}\n\|_{L^1}^{(7-\al)/14}+\|\bar y^{\al}\n\|_{L^1}^{1/2})\\
&\quad+C\|\te\bar x^{-\al/3}\|_{L^{8/\al}}\|\n\bar x^{16\al/(3(4-\al))}\|_{L^1}^{(4-\al)/8}\\
&\le C(\|\na u\|_{L^2}+\|\sqrt\n \dot u\|_{L^2}+\|\sqrt\n\te\|_{L^2}+\|\na\te\|_{L^2}),
\ea\ee
which combined with Lemma \ref{w4} gives that for $q_3=6/\al\geq 6$,
\be\la{lo}\|u\|_{L^{q_3}}\le C\|\na u|x|^{\beta}\|_{L^2}\le C(\|\na u\|_{L^2}+\|\sqrt\n\dot u\|_{L^2}+\|\sqrt\n\te\|_{L^2}+\|\na\te\|_{L^2}).\ee

Then, in view of \eqref{h18}, \eqref{h20}, \eqref{lo}, and  \eqref{t4},  we get
\be\la{gu}\ba
\|u\|_{W^{1,q_3}}
&\le \|u\|_{L^{q_3}}+C \|\rho\dot{u}\|_{L^2}^{(q_3-2)/q_3}\left(\|{\nabla u}\|_{L^2}
   + \|\n\te\|_{L^2}\right)^{{2/q_3}}+C\|\n\te\|_{L^{q_3}}\\
&\le C(\|\na u\|_{L^2}+\|\sqrt\n \dot u\|_{L^2}+\|\sqrt\n\te\|_{L^2}+\|\na\te\|_{L^2}),
\ea\ee
where we have used  the following fact: for any $p\in[4,\infty)$,
\be\la{cc}
\|\n\te\|_{L^p}\le C\|\n\bar x^{a}\|_{L^1}^{4/(p(4+a))}\|\te\bar x^{-{4a/(p(4+a))}}\|_{L^{p(4+a)/a}}
\le C(p)(\|\sqrt\n\te\|_{L^2}\!+\!\|\na\te\|_{L^2}),\ee
due to \eqref{t4}, \eqref{ch}, and \eqref{cc1}.

Now, it follows from \eqref{lo1}, \eqref{c3}, \eqref{xy}, and \eqref{gu} that
\be\la{c9}\ba
&B'_1(t)+\frac{1}{2}B_2\\
&\le C( \|\div u\|_{L^\infty}+\|\na u\|_{L^2}^2\ln(e+\|u\|_{W^{1,q_2}})+1)B_1\\
&\le C( \|\div u\|_{L^\infty}+\|\na u\|_{L^2}^2\ln(B_1^{1/2}+B_2^{1/2})+1)B_1\\
&\le C( \|\div u\|_{L^\infty}+\|\na u\|_{L^2}^2+1)B_1\ln B_1+C\|\na u\|_{L^2}^2\ln (B_2^{1/2})B_1\\
&\le C( \|\div u\|_{L^\infty}+\|\na u\|_{L^2}^2+1)B_1\ln B_1+C\|\na u\|_{L^2}^2(B_1\ln B_1+B_2^{1/2})\\
&\le C( \|\div u\|_{L^\infty}+\|\na u\|_{L^2}^2+1)B_1\ln B_1+\frac{1}{4}B_2,
\ea\ee
where we have used the key  inequality: 
$$ xy\le C(x\ln(e+x)+e^y), \quad\text{for any}\,\, x,y>0.$$ 
Then applying Gronwall's inequality to \eqref{c9} along with \eqref{t3}, \eqref{a2.12}, and \eqref{xy} infers \eqref{a2.1}.
The proof of Lemma \ref{le7} is completed.
\end{proof}

\begin{lemma}\la{le77}
Under the condition (\ref{t3}), it holds that for any $0\le T< T_\ast$,
\be\la{c14}\sup_{0\le t\le T}\|\bar x^a\n\|_{L^1\cap L^\infty}\le C.\ee
\end{lemma}
\begin{proof}
It follows from \eqref{h18}, \eqref{h20}, \eqref{cc}, \eqref{t4}, and \eqref{a2.1} that for any $p\in[4,\infty)$,
\be\la{gff}\ba
\|\na u\|_{L^{p}}\le C(p)( \|\sqrt\n\dot u\|_{L^2}+\|\na\te\|_{L^2}+1 ),
\ea\ee
which together with \eqref{cc1}, \eqref{a2.12}, and \eqref{a2.1} infers that for $0<\eta\le1$,
\be\ba\la{bb111}
\|u\bar x^{-\eta}\|_{L^\infty}&\le C(\eta)(\|u\bar x^{-\eta}\|_{L^{4/\eta}}+\|\na(u\bar x^{-\eta})\|_{L^4})\\
&\le C(\eta)(\|u\bar x^{-\eta}\|_{L^{4/\eta}}+\|\na u\|_{L^4}+\|u\bar x^{-\eta}\|_{L^{8/\eta}}\|\bar x^{-1}\na\bar x\|_{L^{8/(2-\eta)}})\\
&\le C(\eta)(\|\sqrt\n\dot u\|_{L^2}+\|\na\te\|_{L^2}+1).
\ea\ee
Then, in view of \eqref{b11} with $\omega=\bar x^a\n$, using  Gronwall's inequality, \eqref{t3}, \eqref{a2.1}, and \eqref{bb111}, it holds that for any $l\in(1,\infty)$,
\be\la{fd}\sup_{0\le t\le T}\|\bar x^a\n\|_{L^l}\le C.\ee
Note that here the constant $C$ is independent of $l$, therefore \eqref{c14} will be derived after using \eqref{fd} and \eqref{ch}.
The proof of Lemma \ref{le77} is completed.
\end{proof}

\begin{lemma}\la{a113.4} Under the condition (\ref{t3}), it holds that for $0\le T< T_\ast$,
\be\la{ae0} \ba
 \sup_{0\le t\le T}\left(\|\na\te\|_{L^2}^2+\|\sqrt\rho\dot{u}\|_{L^2}^2\right)dx
 + \int_0^T\left(\|\sqrt\n\dot \te\|_{L^2}^2+\|\nabla\dot{u}\|_{L^2}^2\right)dt \le C.\ea\ee
\end{lemma}

\begin{proof}
First, operating $ \dot u^j[\pa/\pa t+\div (u\cdot)]$ to $ (\ref{q1})_2^j$
and integrating the resulting equality over $\r^2 ,$ we obtain after integrating by parts that
\be\la{m4} \ba
 \frac{1}{2}\frac{d}{dt}\int\rho|\dot{u}|^2dx = &-\int\dot{u}^j[\p_jP_t +\div (\p_jPu)]dx +\mu\int\dot{u}^j[\Delta u_t^j +\div (u\Delta u^j)] dx\\
&+ (\mu+\lambda)\int\dot{u}^j[\p_j \div u_t  +\div (u\p_j \div u )] dx   \triangleq \sum_{i=1}^3\tilde J_i. \ea \ee
Note that $ P_t=R\n\dot\te -\div(Pu)$, which along with integration by parts and \eqref{t4} yields 
\be\la{m5} \ba
\tilde J_1 
    & =- \int \dot{u}^j\left(R\p_j (\n \dot\te)-\pa_j\div(Pu)  + \div (\p_jPu)\right) dx \\
    & =- \int \dot{u}^j\left( R\p_j(\n \dot\te)-  \div (P\p_ju)\right) dx\\
    &\le \frac{ 1}{8}\mu \int |\nabla\dot{u}|^2dx+C  \left( \|\sqrt\n \dot\te\|_{L^2}^2+\|\te\na u\|_{L^2}^2 \right) .\ea \ee
Integration by parts leads to
\be\la{m6} \ba
\tilde J_2 
    & = - \mu\int \left(\p_i\dot{u}^j\p_iu_t^j +\Delta u^ju\cdot\nabla\dot{u}^j\right)dx \\
    & = -  \mu\int\left(|\nabla\dot{u}|^2 -\p_i\dot{u}^ju^k\p_k\p_iu^j - \p_i\dot{u}^j\p_iu^k\p_ku^j +
       \Delta u^ju\cdot\nabla\dot{u}^j\right)dx \\
    & = - \mu\int \left(|\nabla\dot{u}|^2 + \p_i\dot{u}^j\p_iu^j\div u - \p_i\dot{u}^j\p_iu^k\p_ku^j - \p_iu^j\p_iu^k\p_k\dot{u}^j
     \right)dx \\
&\le -\frac{7}{8}\mu \int |\nabla\dot{u}|^2dx  + C \|\nabla u\|_{L^4}^4. \ea \ee
Analogously, we have
\be\la{m7}\ba
\tilde J_3&  \le -\frac{7}{8}(\mu+\lambda) \int ({\rm div} \dot u)^2dx + C\|\nabla u\|_{L^4}^4.\ea
\ee
Substituting (\ref{m5})--(\ref{m7}) into (\ref{m4}) and  using \eqref{gff} give that
\be\la{e0}\ba
& \frac{d}{dt}\|\sqrt\rho\dot{u}\|_{L^2}^2
  +\frac{3}{2}\mu \|\nabla\dot{u}\|_{L^2}^2 +\frac{7}{4}(\mu+\lambda)\| {\rm div} \dot u\|_{L^2}^2\\
& \le C_5\|\sqrt\n\dot\te\|_{L^2}^2+C(\|\sqrt\rho\dot{u}\|_{L^2}^4  +\|\na\te\|_{L^2}^4+\|\te\na u\|_{L^2}^2+1).\ea\ee

Next, multiplying $(\ref{q1})_3 $ by $\dot\te$ and integrating the resulting equality over $\r^2 $ infer
\be\la{ee1} \ba
&\frac{\ka}{2}\frac{d}{dt}\|\na\te\|_{L^2}^2+c_\upsilon\int\rho|\dot{\te}|^2dx\\
&=-\ka\int\na\te\cdot\na(u\cdot\na\te)dx+\lambda\int  (\div u)^2\dot\te dx\\
&\quad+2\mu\int |\mathfrak{D}(u)|^2\dot\te dx-R\int\n\te \div u\dot\te dx \triangleq \sum_{i=1}^4\tilde K_i . \ea\ee
Applying the standard $L^2$-estimate to \eqref{q1}$_3$ and using  \eqref{t4} and \eqref{gff}, it holds
\be  \la{op4}\ba
\|\na^2\te\|_{L^2}
&\le C (\|\n\dot \te\|_{L^2}+ \|\na u\|_{L^4}^2+\|\te\na u\|_{L^2}) \\
&\le C(\|\sqrt\n\dot \te\|_{L^2}+\|\sqrt\rho\dot{u}\|_{L^2}^2+\|\na\te\|_{L^2}^2 +\|\te\na u\|_{L^2}+1),\ea\ee
which together with  (\ref{a2.1}) shows that for any $\epsilon\in(0,1),$
\be\la{ee2} \ba
|\tilde K_1|&\le C \int|\na u||\na\te|^2dx\\
     &\le C   \|\na u\|_{L^2}\|\na\te\|_{L^2}\|\na^2\te\|_{L^2}  \\
     &\le \epsilon \|\sqrt\n\dot\te\|^2_{L^2}+C(\epsilon)( \|\na \te\|_{L^2}^4+\|\sqrt\rho\dot{u}\|_{L^2}^4 +\|\te\na u\|_{L^2}^2+1).\ea\ee
Next, integration by parts yields that for any $\eta\in (0,1),$
\be\la{ee3}\ba
\tilde K_2 
=&\lambda\left(\int (\div u)^2 \te dx\right)_t-2\lambda \int \te \div u \div (\dot u-u\cdot\na u)dx\\
&+\lambda\int (\div u)^2u\cdot\na\te dx \\=
&\lambda\left(\int (\div u)^2 \te dx\right)_t-2\lambda\int \te \div u \div \dot udx\\
&+2\lambda\int \te \div u \pa_i u^j\pa_j  u^i dx+ \lambda\int u \cdot\na\left(\te   (\div u)^2 \right)dx\\
\le &\lambda\left(\int (\div u)^2 \te dx\right)_t+\eta\|\na \dot u\|_{L^2}^2+C(\eta)\|\te\na u\|_{L^2}^2+C\|\na u\|_{L^4}^4 .\ea\ee
 Then, similar to (\ref{ee3}),
\be \la{ee5}\ba
\tilde K_3&\le 2\mu\left(\int|\mathfrak{D}(u)|^2 \te dx\right)_t+\eta\|\na \dot u\|_{L^2}^2+C(\eta)\|\te\na u\|_{L^2}^2+C\|\na u\|_{L^4}^4.    \ea\ee
Finally, Cauchy's inequality and \eqref{t4} give
 \be\la{ee39}\ba
| \tilde K_4|    \le \epsilon \|\sqrt\n\dot\te\|^2_{L^2}+C(\epsilon)\|\te\na u\|_{L^2}^2 .  \ea\ee

Substituting  (\ref{ee2})--(\ref{ee39}) into (\ref{ee1}), one has after using  (\ref{gff})  and choosing $\epsilon$ suitably small that, for any $\eta\in (0,1),$
\be\la{e7}\ba
B_3 '(t)+\|\sqrt\n\dot \te\|_{L^2}^2\le  C \eta \|\na\dot u\|_{L^2}^2+C(\eta)\|\te\na u\|_{L^2}^2+C( \|\na \te\|_{L^2}^4+\|\sqrt\rho\dot{u}\|_{L^2}^4+1) ,\ea\ee
where
\be\notag
B_3(t)\triangleq c_\upsilon^{-1}\left(\ka \|\na\te\|_{L^2}^2-2\int F(\na u)\te dx \right).
\ee

Note that  by \eqref{yy} and \eqref{a2.1},  it holds that 
\be\la{ai}\|\na u\bar x^{\al/4}\|_{L^2}\le C(\|\na u\|_{L^2}+\|\na u |x|^{\al/3}\|_{L^2})\le C(\|\sqrt\n\dot u\|_{L^2}+\|\na\te\|_{L^2}+1).\ee
Then by virtue of  \eqref{a2.1}, \eqref{cc1}, \eqref{gff},  and \eqref{ai}, one arrives at that for any $\epsilon_1\in(0,1),$
\be\notag\ba
\int F(\na u)\te dx
&\le C\|\na u\bar x^{\al/4}\|_{L^2}^{1/2}\|\na u\|_{L^2}^{(9-\al)/6}\|\na u\|_{L^4}^{\al/6}\|\te\bar x^{-{\al/8}}\|_{L^{24/\al}}\\
&\le \epsilon_1(\|\na u\bar x^{\al/4}\|_{L^2}^2+\|\na\te\|_{L^2}^2+\|\na u\|_{L^4}^2)+C(\epsilon_1)\\
&\le C\epsilon_1\left(\|\na\te\|_{L^2}^2+\|\sqrt\n\dot u\|_{L^2}^2\right)+C(\epsilon_1),
\ea\ee
which implies that for $C_5$ in  (\ref{e0}), if $\epsilon_1$ is small enough, one has
\be\la{op}\ba
B_4(t)\triangleq \|\sqrt\n\dot u\|_{L^2}^2+(C_5+1) B_3(t)\geq C\left(\|\na\te\|_{L^2}^2+\|\sqrt\n\dot u\|_{L^2}^2\right)-C.\ea\ee

Adding  (\ref{e7}) multiplied by $C_5+1$ to (\ref{e0}),
  we conclude after choosing $ \eta$ suitably small that
\be\la{e8}\ba
B_4'(t) +\mu\|\nabla\dot{u}\|_{L^2}^2+\|\sqrt\n\dot \te\|_{L^2}^2\le C( \|\na \te\|_{L^2}^4 +\|\sqrt\rho\dot{u}\|_{L^2}^4+\|\te\na u\|_{L^2}^2+1).
\ea\ee
Now applying Gronwall's inequality to  \eqref{e8} yields \eqref{ae0} directly, where we have used \eqref{a2.1}, \eqref{op}, and the following  fact:
\be\la{ohh}\ba
\|\te\na u\|_{L^2}^2
&\le C\|\na u\bar x^{\al/2}\|_{L^2}\|\te\bar x^{-{\al/4}}\|_{L^{16/\al}}^2\|\na u\|_{L^2}^{1-{\al/2}}\|\na u\|_{L^4}^{{\al/2}}\\
&\le C(\|\na \te\|_{L^2}^4 +\|\sqrt\rho\dot{u}\|_{L^2}^4+1),
\ea\ee
due to \eqref{cc1}, \eqref{gff}, \eqref{ai}, and \eqref{a2.1}.
 The proof of Lemma \ref{a113.4} is completed.
\end{proof}

\begin{lemma}\la{le8} Under the condition (\ref{t3}), it holds that for $0\le T< T_\ast$,
\be\la{522}\sup_{0\le t\le T}t(\|\sqrt\n\dot\te\|_{L^2}^2+\|\na^2\te\|_{L^2}^2)+\int_0^{T}t\|\na\dot\te\|_{L^2}^2dt\le C.\ee
   \end{lemma}

\begin{proof}
First, by virtue of  \eqref{a2.1}, \eqref{gff}, \eqref{ae0}, \eqref{op4}, \eqref{ai}, and \eqref{ohh}, it holds that for any $p\in[2,\infty)$,
\be\la{glf}\ba
\sup_{0\le t\le T}(\|\na u\|_{L^p}^2+\|\na u\bar x^{\al/4}\|_{L^2}^2)+\int_0^T\|\na^2\te\|_{L^2}^2dt\le C(p).
\ea\ee
Next, applying the operator $\pa_t+\div(u\cdot) $ to (\ref{q1})$_3$  leads to
\be\la{3.96}\ba
&c_\upsilon \n (\pa_t\dot \te+u\cdot\na\dot \te)-\ka \Delta \dot \te\\
&=\ka( \div u\Delta \te  -
\pa_i(\pa_iu\cdot\na \te)- \pa_iu\cdot\na \pa_i\te
)+R\n \te  \pa_ku^l\pa_lu^k -R\n \dot\te \div u\\
&\quad -R\n \te\div \dot u+( \lambda (\div u)^2+2\mu |\mathfrak{D}(u)|^2)\div u +2\lambda ( \div\dot u-\pa_ku^l\pa_lu^k)\div u\\
&\quad+ \mu (\pa_iu^j+\pa_ju^i)( \pa_i\dot u^j+\pa_j\dot u^i-\pa_iu^k\pa_ku^j-\pa_ju^k\pa_ku^i).
\ea\ee
Multiplying (\ref{3.96}) by $\dot \te,$  we deduce from integration by parts, \eqref{glf},  \eqref{cc1}, \eqref{c14}, \eqref{ae0}, and \eqref{a2.1} that
\be\la{3.99}\ba
&\frac{c_\upsilon}{2}\frac{d}{dt}\int \n|\dot\te|^2dx + \ka   \|\na\dot\te\|_{L^2}^2 \\
&\le C  \int|\na u||\dot\te|\left(|\na^2\te|+ |\na\dot u|\right)dx+  C\int  \n\te |\dot\te|(|\na\dot u|+ |\na u|^2)dx\\
&\quad+C\int |\na u|^3|\dot\te|dx+C\int\n  |\dot\te|^2|\div  u|dx +C  \int  |\na u| |\na \te| |\na\dot\te|dx \\
&\le C  \|\na u\bar x^{\al/4}\|^{1/2}_{L^2}\|\na u\|^{1/2}_{L^{12/(6-\al)}}\|\dot \te\bar x^{-{\al/8}}\|_{L^{24/\al}}(\|\na^2\te\|_{L^2}+\|\na\dot u\|_{L^2}) \\
&\quad+C\|\n \bar x\|_{L^6}\|\te \bar x^{-{1/2}}\|_{L^6}\|\dot\te \bar x^{-{1/2}}\|_{L^6}\left(\|\na \dot u\|_{L^2}+\|\na u\|_{L^4}^2\right) \\
&\quad+C\|\na u\bar x^{\al/4}\|_{L^2}\|\na u\|^{2}_{L^{20/(5-\al)}}\|\dot\te\bar x^{-{\al/4}}\|_{L^{10/\al}}
+C\|\div u\|_{L^\infty}\|\sqrt\n\dot\te\|_{L^2}^2 \\
&\quad+C   \|\na u\|_{L^4} \|\na\te\|_{L^2}^{1/2}\|\na^2\te\|_{L^2}^{1/2} \|\na\dot\te\|_{L^2} \\
&\le \frac{\ka}{2}   \|\na\dot\te\|_{L^2}^2 +C(1+\|\div u\|_{L^\infty})\|\sqrt\n\dot\te\|_{L^2}^2+C  (\|\na^2\te\|_{L^2}^2+\|\na\dot u\|_{L^2}^2+1).
\ea\ee
Multiplying \eqref{3.99} by $t$ and applying Gronwall's inequality, \eqref{t3}, (\ref{glf}), and (\ref{ae0}) infer  \eqref{522}.
 The proof of  Lemma \ref{le8} is finished.
\end{proof}

\begin{lemma}\la{le9} Under the condition (\ref{t3}), it holds that for $0\le T< T_\ast$,
\be\la{v1}\sup_{0\le t\le T}\left(\|\na(\bar x^a\n)\|_{L^2\cap L^q}+\|\na^2u\|_{L^2}\right)\le C.\ee
   \end{lemma}
\begin{proof}
First, standard calculation along with \eqref{bb111}, \eqref{ae0}, \eqref{c14}, \eqref{cc1}, \eqref{a2.12}, and \eqref{a2.1} shows that for $\o=\bar x^a\n$,
 \be\la{bb50}\ba
\frac{d}{dt}\|\na \omega\|_{L^q}
&\le C(\|\na u\|_{L^\infty}+\|u\cdot\na \log \bar x\|_{L^\infty})\|\na \omega\|_{L^q}\\
&\quad+C\|\omega\|_{L^\infty}(\|\na^2u \|_{L^q}+\||\na u||\na\log\bar x|\|_{L^q}+\||u||\na^2\log \bar x\|_{L^q})\\
&\le C(\|\na u\|_{L^\infty}+1)\|\na \omega\|_{L^q}+C(\|\na u\|_{W^{1,q}}+1).
\ea\ee

Next, it follows from the standard elliptic estimate,  \eqref{a2.1}, \eqref{ae0}, \eqref{cc1}, \eqref{c14}, \eqref{glf}, \eqref{h19}, and \eqref{h20} that 
\be\la{v3}\ba
\|\na u\|_{W^{1,q}}
&\le C(\|\n \dot u\|_{L^q}+\|\n\na\te\|_{L^q}+\|\na\n\te\|_{L^q}+1)\\
&\le C(\|\n\bar x\|_{L^\infty}\|\dot u\bar x^{-1}\|_{L^q}+\|\na^2\te\|_{L^2}+(\|\na \omega\|_{L^q}+1)\|\te\bar x^{-1}\|_{W^{1,q}}+1)\\
&\le C(\|\sqrt\n \dot u\|_{L^2}+\|\na \dot u\|_{L^2}+\|\na^2\te\|_{L^2}+1)(\|\na\o\|_{L^q}+1),
\ea\ee
and
\be\ba\notag
\|\div u\|_{L^\infty}+\|\curl u\|_{L^\infty}
&\le C(\|G\|_{L^\infty}+\|\n\te\|_{L^\infty}+\|\curl u\|_{L^\infty})\\
&\le C(\| G\|_{W^{1,q}}+\|\curl u\|_{W^{1,q}}+\|\n\bar x\|_{L^\infty}\|\te\bar x^{-1}\|_{W^{1,q}})\\
&\le C(\|\n\dot u\|_{L^q}+\|\na^2\te\|_{L^2}+1)\\
&\le C(\|\sqrt\n \dot u\|_{L^2}+\|\na \dot u\|_{L^2}+\|\na^2\te\|_{L^2}+1).
\ea\ee
This combined with Lemma \ref{lem-bkm} and \eqref{a2.1} leads  to that
\be\ba\la{v6}
\|\na u\|_{L^\infty}
&\le C\left(\|{\rm div}u\|_{L^\infty}+\|\o\|_{L^\infty} \right)\log(e+\|\na^2u\|_{L^q})+C\|\na u\|_{L^2} +C\\
&\le C(\|\sqrt\n \dot u\|_{L^2}+\|\na \dot u\|_{L^2}+\|\na^2\te\|_{L^2}+1)\log(e+\|\na\o\|_{L^q})\\
&\quad+ C(\|\sqrt\n \dot u\|_{L^2}^2+\|\na \dot u\|_{L^2}^2+\|\na^2\te\|_{L^2}^2+1).
\ea\ee

Then let
\be\notag f(t)\triangleq e+\|\na \omega\|_{L^q},\quad\quad f_1(t)\triangleq e+\|\sqrt\n \dot u\|_{L^2}^2+\|\na \dot u\|_{L^2}^2+\|\na^2\te\|_{L^2}^2. \ee
Putting \eqref{v3} and \eqref{v6} into \eqref{bb50}, it holds
\be\notag f'(t)\le Cf_1(t) f(t)\ln f(t),\ee
that is,
\be\notag(\ln(\ln f))'(t)\le Cf_1(t),\ee
which combined with \eqref{glf} and \eqref{ae0}  implies
\be\la{v9}\sup_{0\le t\le T}\|\na(\bar x^a\n)\|_{L^q}\le C.\ee
Similarly, we have
\be\la{v10}\sup_{0\le t\le T}\|\na(\bar x^a\n)\|_{L^2}\le C.\ee
Finally, applying the standard $L^2$-estimate to \eqref{q1}$_2$, we obtain after using \eqref{v9}, \eqref{c14}, \eqref{cc1}, \eqref{a2.1}, and \eqref{ae0} that
\be\notag\ba
\|\na^2u\|_{L^2}&\le C(\|\n\dot u\|_{L^2}+\|\n\na\te\|_{L^2}+\|\na\n\te\|_{L^2})\\
&\le C(\|\sqrt\n\dot u\|_{L^2}+\|\na\te\|_{L^2}+\|\na\n\bar x^{a}\|_{L^q}\|\te\bar x^{-a}\|_{L^{2q/(q-2)}}\\
&\le C,
\ea\ee
which as well as  \eqref{v9} and \eqref{v10} yields \eqref{v1}. The proof of Lemma \ref{le9} is completed.
\end{proof}

With Lemmas \ref{le00}--\ref{le9} in mind, we are  in a position to extend the strong solution $(\n,u,\te)$ beyond $t\geq T_\ast$. Since the generic constant $C$ in Lemmas \ref{le00}--\ref{le9} remains uniformly bounded for all $T<T_\ast$, the functions $(\n,u,\te)(x,T_\ast)\triangleq \lim_{t\rightarrow T_\ast}(\n,u,\te)(x,t)$ satisfy the conditions \eqref{q4} imposed on the initial data at the time $t=T_\ast$. Furthermore, standard arguments yield that $\n\dot u\in C([0,T];L^2)$, which implies
$$\n\dot u(x,T_\ast)=\lim_{t\rightarrow T_\ast}\n\dot u(x,t)\in L^2.$$
Therefore, one has
\be\notag
-\mu\Delta u-(\mu+\lambda)\na\div u+R\na(\n\te)|_{t=T_\ast}=\n^{1/2}(x,T_\ast)g(x),
\ee
with
\be\notag\ba
g(x)\triangleq
\begin{cases}
-\n^{-1/2}(x,T_\ast)\n\dot u(x,T_\ast),&\text{for}\,x\in\{x|\n(x,T_\ast)>0\},\\
0,&\text{for}\,x\in\{x|\n(x,T_\ast)=0\},
\end{cases}
\ea\ee
satisfying $g\in L^2$ owing to Lemma \ref{a113.4}. Then, $(\n,u,\te)(x,T_\ast)$ also satisfy \eqref{co1}. Consequently, one can take $(\n,u,\te)(x,T_\ast)$ as the initial data and apply Theorem \ref{th1} to extend the local strong solution beyond $T_\ast$. This contradicts the assumption on $T_\ast$. We thus finish the proof of Theorem \ref{th2}.

\section{Appendix}\la{sec5}
In this appendix, we present the proof of Lemma \ref{th3}. 

Before the formal argument, we first introduce some notations used frequently later. For $1\le p\le \infty$ and integer  $k\geq 0$, we denote
$$L^p\triangleq L^p(\O),\quad W^{k,p}\triangleq W^{k,p}(\O),\quad H^{k}\triangleq W^{k,2},\quad \int f dx\triangleq \int_{B_r} f dx.$$

Now, for any fixed $\de>0$, we consider the approximate problem (\ref{a1})--(\ref{bbb}). For $r>4N_2\geq 4$, we  assume the initial data   $(\n_0^r,u_0^r,\te_0^r)$ satisfy, in addition to \eqref{2.1}, that
\be\la{t1}\int_{B_{N_2}}\n_0^r dx\geq 1/2.\ee
Then, Lemma \ref{w1} implies that there exists some $T_r>0$ such that the initial-boundary-value problem (\ref{a1})--(\ref{bbb})    has a unique classical solution $(\n^r,u^r,\te^r)$    on $B_r\times(0,T_r]$ satisfying \eqref{mn6} and \eqref{mn5} with $(\n,u,\te)$ replaced by $(\n^r,u^r,\te^r)$.    The superscript $r$ hereinafter is omitted.

Next, we have the following key uniform a priori estimates with respect to $r$ on $\psi$ defined in \eqref{bb0}, which guarantees the extension of solutions in $B_r$.
\begin{pro}\la{pr1} 
For fixed $\de>0$ and the initial data $(\n_0,u_0,\te_0)$ satisfying (\ref{2.1}), (\ref{co1}),  and (\ref{t1}), assume that $(\n,u,\te)$ is a smooth solution to the problem (\ref{a1})--(\ref{bbb})  on $\O\times(0,T_r]$ obtained in Lemma \ref{w1}. 
 We define
\be\la{hk} T_r^\ast= \sup\left\{T\le T_r\left|\sup_{0\le t\le T}\|\sqrt\n u\|_{L^2}^2\le M_2\right.\right\},\ee
for some positive constant $M_2$ determined later, which is  depending only on $\mu$,  $\lambda$, $R,$ $c_\upsilon,\,\ka,\,q,\,a,\,\eta_0,$ $N_2$,   $\psi_{0}$, and $\de$.
Then there exist   positive constants  $T_\de(\le T_r^\ast)$ and $\tilde M_2$ depending only on $\mu$,  $\lambda$, $R,$ $c_\upsilon,\,\ka,\,q,\,a,\,\eta_0,$ $N_2$,  $\psi_0$, and $\de$ 
 but independent of $r$ such that
\be\ba\notag
\psi(T_\de)\le M_2/2,
\ea\ee
\be\notag\sup_{0\le t\le T_\de}\|\na^2u\|_{L^2}^2+\int_0^{T_\de}\left(\|\na u_t\|_{L^2}^2+\|\sqrt\n\te_t\|_{L^2}^2+\|\na^2\te\|_{L^2}^2+\|\na^2 u\|_{L^{q}}^2\right)dt\le \tilde M_2,\ee
where $\psi_0$ is defined in (\ref{000}) 
and $g$ given by  (\ref{co1}) in fact is
$$g=\n_0^{-1/2}(-\mu\Delta u_0-(\mu+\lambda)\na\div u_0+R\na(\n_0\te_0)).$$
\end{pro}

In the following section, we assume  $C$ and $A$  denote some generic constants depending only on $\mu,\,\lambda,\,R,\,c_\upsilon,\,\ka$, $q$, $a$, $\eta_0$, $N_2$, and $\psi_0$, but independent of  $\de$ and $r$.

To obtain Proposition \ref{pr1}, we need the following lemmas. We start with  some elementary weighted estimates.
\begin{lemma}\la{lee00}
Under the conditions of Proposition \ref{pr1}, 
 there exists a positive time $T_2=T_2(  M_2,N_2,\|\n_0\|_{L^1})$ independent of  $r$ such that for any $0<t\le T_2$,  $v\in \tilde D^{1,2}(\O)$,  $\ve>0$, and $\eta\in(0,1],$ 
\be \la{dd1}\ba
\|v\bar x^{-1}\|_{L^2}^2+\|v\bar x^{-\eta}\|_{L^{(2+\ve)/\eta} }^2
\le C(\epsilon,\eta) \|\sqrt{\n}v\|_{L^2}^2 +C(\epsilon,\eta) (\|\n\|_{L^\infty}+1) \|\na v\|_{L^2 }^2.
\ea\ee
Moreover, for any $0<t\le T_2$,   $\ve>0$, and $\eta\in(0,1],$  the following estimates hold
\be\la{x2}\ba
\|u\bar x^{-\eta}\|_{L^{(2+\ve)/\eta}}\le C(\epsilon,\eta)\psi_1^{1/2}\psi_2^{1/2},\quad \|\n^{\eta}u\|_{L^{(2+\ve)/\eta}}\le C(\epsilon,\eta)\psi_1^{1/2}\psi_2^{1/2+\eta},\ea\ee
\be\la{x3}\ba
\|\te\bar x^{-\eta}\|_{L^{(2+\ve)/\eta}}
\le C(\epsilon,\eta)\psi_1^{{1/2}}\psi_2^{1/2}\psi_3^{1/2},\quad \|\n^{\eta}\te\|_{L^{(2+\ve)/\eta}}
\le C(\epsilon,\eta)\psi^{1+\eta},\ea\ee
with $\psi_1$, $\psi_2$, and $\psi_3$ respectively defined in (\ref{b0}), (\ref{bq0}), and (\ref{bp0}).
\end{lemma}
\begin{proof}
It follows from $\eqref{a1}_1$ that for any $t\in(0,T_r]$,
$\|\n\|_{L^1}(t)=\|\n_0\|_{L^1},$
which together with \eqref{t1} and  \eqref{hk} yields
$$\inf_{0\le t\le T_2}\int_{B_{2N_2}}\n dx\geq \frac{1}{4},$$
for some positive constant $T_2$, which depends only on $M_2$, $N_2$, and  $\|\n_0\|_{L^1}$ and satisfies $T_2\le T_r^\ast$ for any $r>1$.  Then, combining this with Lemma \ref{w5} leads to \eqref{dd1}--\eqref{x3}. 
\end{proof}

Next, applying the standard elliptic estimates to \eqref{a1}$_2$,  \eqref{a1}$_3$ and using Sobolev inequality, \eqref{dd1}--\eqref{x3}, one gets  the following lemma immediately (see \eqref{b04}, \eqref{b111}, \eqref{b05}--\eqref{b112}, and \eqref{der} for the detailed proof). 

\begin{lemma}\la{jia}
Let $T_2$ be as in Lemma \ref{lee00}. Then for $q$ as in Theorem \ref{th1} and $\eta\in(0,1],$ we have
\be\la{b044}\|\na u\|_{H^1}\le C\psi^A,\quad\|u\bar x^{-\eta}\|_{L^\infty}\le C(\eta)\psi^A,\ee
\be\la{b110}\|\na \te\|_{H^1}\le C\psi^A(\|\sqrt\n\te_t\|_{L^2}+1),\quad\|\te\bar x^{-\eta}\|_{L^\infty}\le C(\eta)\psi^A(\|\sqrt\n\te_t\|_{L^2}+1),\ee
\be\la{bqq}\|\na u\|_{W^{1,q}}\le C\psi^A(\|\na u_t\|_{L^2}+\|\sqrt\n\te_t\|_{L^2}+1),\ee
\be\la{bbq}\|\na^2\te\|_{H^{1}}\le C\psi^A(\|\sqrt\n\te_t\|_{L^2}+\|\na u_t\|_{L^2}+\|\na\te_t\|_{L^2}+1).\ee
\end{lemma}

Now, we can establish the following basic a priori estimates, where the analogue proof as the ones in Subsection \ref{sec3.2} shall be omitted for simplicity. In fact, in view of the integrability of temperature being available  and \eqref{dd1}--\eqref{bqq} being  the same as the ones in Subsection \ref{sec3.2},  only the terms involving the dissipative function $F(\na u)$ in Lemmas \ref{le01} and  \ref{le4} need to be re-estimated.

\begin{lemma}\la{lee4}
 Let $T_2$ be as in Lemma \ref{lee00}. Then there exist positive  constants $C$, $C(\de)$, and $A\geq1$ such that for any $0<T\le T_2$,
 \be\ba\notag
\psi_1(T)\le C+C(\de)\int_0^T\psi^Adt.
\ea\ee
\end{lemma}
\begin{proof}
Note that  \eqref{nau} still holds here, i.e.,
$$\|\na u\|_{L^6}\le C\psi_1\psi_2^{3/2}\psi_3^{1/2},$$
which implies
\be\la{terms}\ba
\left|\int F(\na u)\te dx\right|
\le C\|\na u\|_{L^2}^{3/2}\|\na u\|_{L^6}^{1/2}\|\te\|_{L^2}^{1/3}\|\te \|_{H^1}^{2/3}\le C(\de)\psi_1^{5/4}\psi_2^{3/4}\psi_3^{3/4}.
\ea\ee

Replacing \eqref{term} with \eqref{terms} in \eqref{on} and proceeding as the same calculation as in Lemma \ref{le01},  we can finish the proof of Lemma \ref{lee4} smoothly.
\end{proof}

\begin{lemma}\la{lee6}
Let  $q$  be as in Theorem \ref{th1} and $T_2$ as in Lemma \ref{lee00}. Then for any $p'>2$, there exist  positive constants $C$, $C(p')$, and $A\geq 1$ such that for any $0<T\le T_2$,
\be\notag
\psi_2(T)\le C+C(p')\psi_2^{q/(2q+p'q-2p')}(T)\left(1+\int_0^T\psi^Adt\right).\ee
\end{lemma}

\begin{lemma}\la{lee3}
Let $T_2$ be as in Lemma \ref{lee00}. Then there exist positive  constants  $C(\de)$ and $A\geq1$ such that for  any $0<T\le T_2$,
\be\ba\notag
\psi_3(T)\le C(\de)\left(1+\int_0^T\psi^A dt\right).
\ea\ee
\end{lemma}
\begin{proof}
To begin with, we re-estimate the term $J_3$ in \eqref{b33}. It follows from  \eqref{b044} that
\be\notag\ba
\left|\int F(\na u)_t\te dx\right|
\le C\|\na u_t\|_{L^2}\|\na u\|_{L^2}^{1/2}\|\na u\|_{H^1}^{1/2}\|\te\|_{L^2}^{1/2}\|\te\|_{H^1}^{1/2}
\le C(\de)\psi^A\|\na u_t\|_{L^2},
\ea\ee
which is the same as the one in Lemma \ref{le4} except that the constant $C$ is dependent on $\de$.
Then following the same procedures as in Lemma \ref{le4} shows that
\be\ba\notag
&C\frac{d}{dt}\int(\ka|\na\te|^2+\de|\te|^2-2 F(\na u)\te+\ka\rho^2 |\na\te|^2+\de\n^2|\te|^2)\\
&+\frac{d}{dt}\int\rho|u_t|^2dx+\frac{1}{2}\int(\mu|\na u_t|^2+\n|\te_t|^2+\n^3|\te_t|^2)dx\le C(\de)\psi^A,
\ea\ee
which along with \eqref{terms} and Lemmas \ref{lee4}--\ref{lee6}  leads to that for any $p'>2$,
\be\ba\notag
&\psi_3(T)+\int_0^T(\|\na u_t\|_{L^2}^2+\|\sqrt\n\te_t\|_{L^2}^2)dt\\
&\le C+C(\de)\int_0^T\psi^A dt+C(\de)\psi_1^{5/4}(T)\psi_2^{3/4}(T)\psi_3^{3/4}(T)\\
&\le \left(C(\de)+C(\de,p')\psi_2^{3q/(4(2q+p'q-2p'))+3/4}(T)\right)\left(1+\int_0^T\psi^Adt\right)^2.
\ea\ee
Then,  choosing suitable large $p'$ and  using Young's  inequality  conclude the proof of Lemma \ref{lee3}.
\end{proof}

\begin{lemma}\la{lee2}
Let   $T_2$ be as in Lemma \ref{lee00}. Then there exist  positive constants $C$ and $A\geq 1$ such that for any $0<T\le T_2$,
\be\notag\sup_{0\le t\le T}\|\na(\bar x^a\n)\|_{ L^{q}}\le C+C\int_0^T\psi^A dt.\ee
\end{lemma}

{\bf Proof of Proposition \ref{pr1}.}
For $T_2$ as in Lemma \ref{lee00},  Lemmas \ref{lee4}--\ref{lee2} lead to that for any $0<T\le T_2$,
\be\ba\notag
\psi(T)\le C(\de)\left(1+\int_0^T\psi^A dt\right).\ea\ee
Denote $ M_2=4C(\de)$, then for $T_\de\triangleq \min\{T_2,  M_2^{-A}\}$, it holds
\be\ba\notag\psi(T_\de)\le M_2/2,\ea\ee
which together with \eqref{b044}--\eqref{bqq}  completes the proof of Proposition \ref{pr1}.
\thatsall

\begin{lemma}\la{lee5}
Let $T_\de$ be as in Proposition \ref{pr1}. Then there exists  a positive constant $C(\de)$  such that 
\be\ba\notag
&\sup_{0\le t\le T_\de}t\left(\|\sqrt\n\te_t\|_{L^2}^2+\|\na^2\te\|_{L^2}^2\right)\\
&+\int_0^{T_\de}t(\|\na\te_t\|_{L^2}^2+\de\|\te_t\|_{L^2}^2+\|\na^2\te\|_{H^{1}}^2)dt\le C(\de).
\ea\ee
\end{lemma}
\begin{proof}
Actually, we  only need to estimate the term $L_5$ in \eqref{b54}. By virtue of  Proposition \ref{pr1},  one gets for any $\epsilon\in(0,1)$,
\be\ba\notag
\left|\int F(\na u)_t\te_t dx\right|
   &\le C\|\na u_t\|_{L^2}\|\na u\|_{L^4}\|\te_t\|_{L^2}^{1/2}\|\te_t\|_{H^1}^{1/2}\\
     &\le \epsilon(\|\na\te_t\|_{L^2}^2+\de\|\te_t\|_{L^2}^2)+C(\epsilon,\de)\|\na u_t\|_{L^2}^2.
\ea\ee
Then we finish the remaining proof of Lemma \ref{lee5} in the similar way as Lemma \ref{le6}.
\end{proof}

Now, we are in a position to prove  Lemma \ref{th3} formally.

{\bf Proof of Lemma \ref{th3}.}  Let $(\n_0\geq 0,u_0,\te_0)$ be the initial data to the Cauchy problem (\ref{a1}) (\ref{q2}) (\ref{q3}) satisfying conditions \eqref{q4}, \eqref{co1}, and $\te_0\in L^2(\r^2)$.  Without loss of generality, we assume that $\|\n_0\|_{L^1(\r^2)}=1,$ which implies there exists a  constant $N_2\geq 1$ such that
\be\notag\int_{B_{N_2}}\n_0 dx\geq 3/4.\ee

To begin with, we construct $\n_0^r=\hat\n_0^r+r^{-1}e^{-|x|^2}$, where $0\le \hat\n_0^r\in C_0^\infty(\O)$ satisfies that
\be\notag\int_{B_{N_2}}\hat\n_0^r dx\geq 1/2\quad \text{and} \quad \lim_{r\rightarrow \infty}\|\bar x^a\hat\n_0^r- \bar x^a\n_0\|_{L^1(\r^2)\cap H^1(\r^2)\cap W^{1,q}(\r^2)}=0,\ee
which leads to
\be\la{52}\ba\int_{B_{N_2}}\n_0^r dx\geq 1/2&\quad \text{and} \quad \lim_{r\rightarrow \infty}\|\bar x^a\n_0^r- \bar x^a\n_0\|_{L^1(\r^2)\cap H^1(\r^2)\cap W^{1,q}(\r^2)}=0.
\ea\ee
Next, we choose $\te_0^r$ as the unique smooth solution to the following elliptic problem:
\be\notag\ba
\begin{cases}
-\Delta \te_0^r+\te_0^r=-\Delta (\te_0\ast j_{r^{-1}})+\te_0\ast j_{r^{-1}}, \quad &\text{in}\,\O,\\
\na\te_0^r\cdot n=0,&\text{on}\,\p\O,
\end{cases}
\ea\ee
with $j_{r^{-1}}$ being the standard mollifying kernel of width $r^{-1}$.
According to  the procedures in \cite{liliang},  we obtain after using \eqref{q4} and $\te_0\in L^2(\r^2)$ that
\be\la{qg}\|\te_0^r\|_{H^1}\le C,\quad\lim_{r\rightarrow \infty}\|\te_0^r-\te_0\|_{H^1(\r^2)}=0.\ee
Similarly, we consider the unique smooth solution $u_0^r$ of the following elliptic problem:
\be\notag\ba
\begin{cases}
-\mu\Delta u_0^r-(\mu+\lambda)\na\div u_0^r+R\na(\n_0^r\te_0^r)+\n_0^r u_0^r=(\n_0^r)^{1/2}h^r,\quad &\text{in}\,\O,\\
u_0^r=0,&\text{on}\,\p\O,
\end{cases}
\ea\ee
where $h^r=(\sqrt{\n_0}u_0+g)\ast j_{r^{-1}}$.  Then, by \eqref{q4}, \eqref{52}, and \eqref{qg}, one has
\be\la{d47} \|\na u^r_0\|_{H^1(B_r)}^2+\|\sqrt{\n_0^r}u_0^r\|_{L^2(B_r)}^2\le C,\ee
\be\ba\la{d48} &\lim_{r\rightarrow\infty}\left(\|\na u_0^r-\na u_0\|_{H^1(\r^2)}+\|\sqrt{\n_0^r}u_0^r-\sqrt{\n_0}u_0\|_{L^2(\r^2)}\right)=0.
\ea\ee

Now, for the approximate initial data $(\n_0^r,u_0^r,\te_0^r)$ satisfying \eqref{52}--\eqref{d48}, Lemma \ref{w1} thus infers that there exists a smooth solution  $(\n^r,u^r,\te^r)$ to the  problem (\ref{a1})--(\ref{bbb})  on $\O\times(0,T_r]$ satisfying \eqref{mn6} and \eqref{mn5} with $(\n,u,\te)$ being replaced by $(\n^r,u^r,\te^r)$. Moreover, for the initial data $(\n_0^r,u_0^r,\te_0^r)$, $g$ defined in \eqref{co1} is replaced by $g^{r}\triangleq h^r-\sqrt{\n_0^r}u_0^r$ satisfying
\be\notag\|g^{r}\|_{L^2(B_r)}\le C,\ee
due to $g\in L^2(\r^2)$, \eqref{q4}, and \eqref{d47}.
Therefore, by virtue of Proposition \ref{pr1},  there exist positive constants $T_\de$, $M_2$, and $C(\de)$ independent of $r$ such that the estimates in Proposition \ref{pr2} and Lemma \ref{lee5} hold for $(\n^r,u^r,\te^r)$. 

Denote
$$\tilde\n^r=\n^r\varphi_r,\quad\tilde u^r=u^r\varphi_r,\quad\tilde\te^r=\te^r\varphi_r,$$
where $(\n^r,u^r,\te^r)$ is extended  by zero to $\r^2$    and  $\varphi_r$ is defined in \eqref{we}.
Then  Proposition \ref{pr1}, Lemma \ref{lee5}, and Poincar$\acute{e}$ inequality conclude that the sequence $(\tilde\n^r, \tilde u^r, \tilde\te^r)$ converges weakly, up to the extraction of subsequences, to some limit $(\n, u,\te)$ satisfying \eqref{bb}.
Next, for any function $\phi\in C_0^\infty(\r^2\times[0,T_\de))$, we take $\phi(\varphi_r)^4$ as text function in the initial-boundary value problem (\ref{a1})--(\ref{bbb})  with the initial data $(\n_0^r,u_0^r,\te_0^r)$. Then let $r\rightarrow \infty$, it follows from  standard arguments that $(\n,u ,\te)$ in fact is a strong solution of \eqref{a1} \eqref{q2} \eqref{q3} on $\r^2\times(0,T_\de]$.  The proof of Lemma \ref{th3} is finished.
\thatsall

\end{document}